\documentclass[a4paper,UKenglish,cleveref, autoref, thm-restate]{lipics-v2021}
\hideLIPIcs  %uncomment to remove references to LIPIcs series (logo, DOI, ...), e.g. when preparing a pre-final version to be uploaded to arXiv or another public repository

\usepackage{bm}
\usepackage{graphicx}

\title{Quasi-Optimal Interpolation of Gradients and Vector-Fields on Protected Delaunay Meshes in $\mathbb{R}$\lowercase{$^d$}} %TODO Please add

\titlerunning{Quasi-Optimal Interpolation on Protected Delaunay Meshes in $\mathbb{R}$\lowercase{$^d$}} %TODO optional, please use if title is longer than one line

\author{David M. Williams\footnote{Distribution Statement A: Approved for public release. Distribution is unlimited.}}{Department of Mechanical Engineering, Pennsylvania State University, University Park, Pennsylvania, 16801, United States of America}{david.m.williams@psu.edu}{https://orcid.org/0009-0009-9137-1416}{This researcher was funded by the United States Naval Research Laboratory (NRL) under grant number N00173-22-2-C008. In turn, the NRL grant itself was funded by Steven Martens, Program Officer for the Power, Propulsion and Thermal Management Program, Code 35, in the United States Office of Naval Research, and by Saikat Dey, Acoustics Division Theoretical and Numerical Techniques Section Head, NRL.}

\author{Mathijs Wintraecken}{INRIA Sophia Antipolis - Méditerranée, 06902 Sophia Antipolis, France, and Université Côte d’Azur, France}{mathijs.wintraecken@inria.fr}{https://orcid.org/0000-0002-7472-2220}{This researcher was funded by the French National Research Agency (ANR) under grant StratMesh and the welcome package from IDEX of the Universit{\'e} C{\^o}te d'Azur. }

\authorrunning{D. M. Williams and M. Wintraecken} %TODO mandatory. First: Use abbreviated first/middle names. Second (only in severe cases): Use first author plus 'et al.'

\Copyright{D. M. Williams and M. Wintraecken} %TODO mandatory, please use full first names. LIPIcs license is "CC-BY";  http://creativecommons.org/licenses/by/3.0/

% mandatory: Please choose ACM 2012 classifications from https://www.acm.org/publications/class-2012 or https://dl.acm.org/ccs/ccs_flat.cfm . E.g., cite as "General and reference $\rightarrow$ General literature" or \ccsdesc[100]{General and reference~General literature}. 

 %TODO mandatory; please add comma-separated list of keywords

\nolinenumbers %uncomment to disable line numbering

\EventEditors{John Q. Open and Joan R. Access}
\EventNoEds{2}
\EventLongTitle{41st International Symposium on Computational Geometry (SoCG 2025)}
\EventShortTitle{SoCG 2025}
\EventAcronym{SoCG}
\EventYear{2025}
\EventDate{June ???, 2025}
\EventLocation{Kanazawa, Japan}
\EventLogo{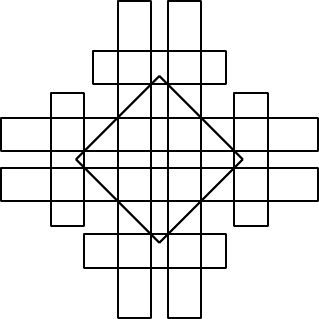}
\SeriesVolume{????}

\begin{document}

\maketitle

%TODO mandatory: add short abstract of the document
\begin{abstract}
    There are very few mathematical results governing the interpolation of functions or their gradients on Delaunay meshes in more than two dimensions. Unfortunately, the standard techniques for proving optimal interpolation properties are often limited to triangular meshes. Furthermore, the results which do exist, are tailored towards interpolation with piecewise linear polynomials. In fact, we are unaware of \emph{any} results which govern the high-order, piecewise polynomial interpolation of functions or their gradients on Delaunay meshes. In order to address this issue, we prove that quasi-optimal, high-order, piecewise polynomial gradient interpolation can be successfully achieved on 
    \emph{protected} Delaunay meshes.  In addition, we generalize our analysis beyond gradient interpolation, and prove quasi-optimal interpolation properties for sufficiently-smooth vector fields. Throughout the paper, we use the words `quasi-optimal', because the quality of interpolation depends (in part) on the minimum thickness of simplicies in the mesh. Fortunately, the minimum thickness can be precisely controlled on protected Delaunay meshes in $\mathbb{R}^d$. Furthermore, the current best mathematical estimates for minimum thickness have been obtained on such meshes. In this sense, the proposed interpolation is optimal, although, we acknowledge that future work may reveal an alternative Delaunay meshing strategy with better control over the minimum thickness. With this caveat in mind, we refer to our interpolation on protected Delaunay meshes as quasi-optimal.
\end{abstract}

\section{Introduction}

The primary purpose of this article is to motivate the construction of \emph{protected} Delaunay meshes in higher dimensions, $d>2$. In order to avoid confusion, we will refer to traditional Delaunay meshes which satisfy an empty-hypersphere criterion as \emph{standard} Delaunay meshes. In contrast, a \emph{protected} Delaunay mesh satisfies a modified empty-hypersphere criterion, where the hypersphere of each simplex is augmented by a spherical-buffer region. This buffer region is formed by taking the original hypersphere with radius $R$, and subtracting its volume from an augmented hypersphere with radius $R + r$. Here, we insist that $r \geq \delta$, for some $\delta > 0$. In this context, the quantity $\delta$ is called the \emph{protection}. Broadly speaking, our goal is to make $\delta$ as large as possible, as this has two positive ramifications: (i) it increases the minimum size of slivers in the Delaunay mesh, and (ii) it reduces the sensitivity of the Delaunay mesh to the locations of its points. In this work, we are primarily interested in the first benefit, as we can directly reduce the errors of gradient or vector-field interpolation by increasing the minimum size of slivers. Of course, in a \emph{standard} Delaunay mesh without protection, the thickness of a sliver can become arbitrarily close to zero, (for $d > 2$). Fortunately, in cases where the protection is non-zero, we obtain fatter simplices whose thickness is bounded away from zero. This fact was established in the pioneering work of Boissonnat, Dyer, and Ghosh~\cite{boissonnat2013stability}. More precisely, they showed that slivers naturally arise from pathological configurations of $d+2$ vertices which are nearly co-spherical. These pathological configurations can be avoided by carefully perturbing the mesh points to facilitate the construction of a protected Delaunay mesh. A computationally expensive procedure for perturbing the points has been proposed in~\cite{boissonnat2014delaunay}. In addition, a detailed summary of protected Delaunay meshes in $\mathbb{R}^d$ appears in~\cite{boissonnat2018geometric}.

Now, having established the concept of a protected Delaunay mesh, let us expand our discussion in order to present some broader perspectives on the entire field of Delaunay meshing. Generally speaking, standard Delaunay meshes have a very good reputation among scientists and engineers. This positive reputation has been documented in many places, including the excellent textbooks of Cheng, Dey, and Shewchuk~\cite{cheng2013delaunay}, and Borouchaki and George~\cite{borouchaki2017meshing}. However, in our opinion, the reputation of these meshes is mostly based on their optimality properties in $\mathbb{R}^2$. Unfortunately, there is a weaker justification for using standard Delaunay meshes in higher dimensions. In particular, Delaunay triangulations of nets (well-spaced point sets) still contain slivers in dimensions higher than $2$, while in 2D the quality of simplices in a Delaunay triangulation of a net is lower bounded (as mentioned previously).  
In the next section, we will provide a short review of optimality properties of standard Delaunay meshes, and identify areas for potential improvement.  

\subsection{Background: Optimality of Standard Delaunay Meshes}

There are many optimality results for standard Delaunay triangulations in $\mathbb{R}^2$. For example, Sibson~\cite{sibson1978locally} proved that a standard Delaunay triangulation of a point set $S$ is guaranteed to maximize the minimum-interior angle of its triangles, relative to any other triangulation of the same points. In addition, Musin~\cite{musin1993delaunay,musin1997properties} proved that the standard Delaunay triangulation of $S$ minimizes the average circumradius of triangles in the triangulation. In a similar fashion, Lambert~\cite{lambert1994delaunay} proved that the standard Delaunay triangulation of $S$ minimizes the average inradius. Furthermore, Musin~\cite{musin1995index} proved that the standard Delaunay triangulation minimizes the harmonic index functional, where the functional is the sum over each triangle of the squared edge lengths divided by the triangle area. In more recent work, Musin~\cite{musin2010optimality} conjectured that the mean radius functional and the $D$ functional are minimized on standard Delaunay triangulations. Here, the mean radius functional is an area-weighted sum of the squared circumradii, and the $D$ functional is an area-weighted sum of the squared distance between the barycenter and circumcenter of each triangle. Following this work, Edelsbrunner et al.~\cite{edelsbrunner2017voronoi} proved that the mean radius functional is minimized on standard Delaunay triangulations. In addition, they showed that the Voronoi functional is maximized on the same triangulations. For the sake of brevity, the precise formulation of the Voronoi functional will not be described here; the interested reader is encouraged to consult~\cite{akopyan2009extremal} for details. In addition, a detailed summary of the optimality properties of standard Delaunay triangulations appears in Sierra's thesis~\cite{sierra2021optimality}. 

Most of the results above, govern the shape-regularity of triangles in a standard Delaunay triangulation. Essentially, these results ensure that the triangles resemble equilateral triangles, as much as possible. 

One may also obtain results which directly predict the interpolation or approximation accuracy of a standard Delaunay triangulation. In particular, Rippa~\cite{rippa1990minimal} and Powar~\cite{powar1992minimal} proved that the piecewise linear interpolations of $H^1$ functions have minimal \emph{roughness} on a standard Delaunay triangulation. Here, the roughness of the linear interpolation is defined as the integral of the squared magnitude of the gradient. This quantity naturally arises when a classical finite element method is applied to elliptic problems in $\mathbb{R}^2$. In~\cite{rippa1990minimum}, Rippa and Schiff leveraged the roughness result of~\cite{rippa1990minimal}, and proved that standard Delaunay triangulations minimize the solution error for simple elliptic problems. Thereafter, Shewchuk~\cite{shewchuk2002} performed an exhaustive study of  piecewise linear interpolation on generic triangular and tetrahedral meshes. Here, Shewchuk presented techniques for improving interpolation error on meshes that are not necessarily Delaunay.

Regrettably, there are few optimality properties for standard Delaunay meshes in $\mathbb{R}^{d}$ when $d > 2$. In~\cite{rajan94optimality}, Rajan proved that the  standard Delaunay mesh minimizes a functional of the weighted sum of the squared edge lengths, (see Theorem 1 of~\cite{rajan94optimality}). In addition, Rajan proved that the standard Delaunay mesh minimizes the maximum, min-containment radius of simplices in the mesh, (see Theorem 2 of~\cite{rajan94optimality}). These results are somewhat abstract, but they can easily be clarified with appropriate examples. In particular, Rajan's functional in $\mathbb{R}^4$ is the sum over each 4-simplex of the hypervolume-weighted squared edge lengths of the simplex. Furthermore, the min-containment radius of a 4-simplex corresponds to the hypersphere of minimum radius which contains the simplex. This latter quantity is often called the \emph{enclosing} hypersphere or ball~\cite{cheng2013delaunay}.

Following the work of Rajan, Waldron~\cite{waldron1998error} proved that a standard Delaunay mesh minimizes the infinity error for the piecewise linear interpolation of a multivariate function with pointwise-bounded second derivatives, (see Theorem 3.1 of~\cite{waldron1998error}). We note that Waldron's work implicitly leverages the min-containment radius result of Rajan~\cite{rajan94optimality}. 

In addition, one can prove that the standard Delaunay mesh provides optimal piecewise linear interpolation of the quadratic function, $\left|\bm{x} \right|^{2} + \bm{a}\cdot \bm{x} + b$, where $\bm{x}$ is a generic point in $\mathbb{R}^d$, $\bm{a}\in\mathbb{R}^d$, and $b \in \mathbb{R}$~\cite{cheng2013delaunay}. For this function, the standard Delaunay mesh minimizes the error in the $L_p$-norm for $p \geq 1$. This result has been leveraged in order to construct objective functions for \emph{optimal Delaunay triangulations}, (ODTs)---see the pioneering work of Chen and Xu~\cite{chen2004optimal}. In particular, ODTs are defined based on an energy functional or objective function which takes a mesh as input. Let  $f: \bm{x} \rightarrow |\bm{x}|^2$ be the parabola. 
A given mesh $\mathcal{T}$, induces a piecewise linear interpolation $f_{\mathrm{pl}}$, which coincides with $f$ on the vertices of $\mathcal{T}$ and is a linear interpolation on each simplex. The energy functional (objective function) is now defined as the integrated error that the PL-interpolation makes, that is $\mathcal{F}_{\textrm{ODT}} (\mathcal{T}) = \left\|f - f_{\mathrm{pl}}\right\|_{L^{1}(\Omega)}^{2}$. An optimal Delaunay triangulation is a mesh that minimizes this functional within the class of meshes with the same number of vertices. We stress that this means that both the position of the vertices and the combinatorics of the mesh are not fixed in this optimization.
%DW: Thanks for clarifying the matter!
However, because for a fixed set of vertices the Delaunay triangulation minimizes $\mathcal{F}_{\textrm{ODT}} (\mathcal{T})$, the result is always a Delaunay triangulation (assuming the vertices are in general position).  
%If the mesh in question is already Delaunay, then the ODT procedure will usually make the mesh more isotropic, as the minimum energy of the objective function is achieved asymptotically for isotropic elements, (see~\cite{nadler1986piecewise}).
We note that other objective functions can be used as alternatives to $\mathcal{F}_{\textrm{ODT}} (\mathcal{T})$, including functions which minimize the hypervolume-weighted sum of the edge lengths~\cite{chen2011efficient}, in accordance with Rajan's result~\cite{rajan94optimality}. From our perspective, the only issues with the ODT approach are, (a) the lack of theoretical guarantees on the minimum size of slivers, and (b) the inherent focus on piecewise linear interpolation.     

Finally, we note that Musin~\cite{musin1997properties} proved that a parabolic functional is minimized on standard Delaunay meshes. This functional consists of a hypervolume-weighted sum over each simplex of the squared vertex locations. While this is an interesting result, we are presently unaware of how it can be used in practical applications. 

\subsection{Summary of Existing Literature and New Contributions}

The best work on the optimality of standard Delaunay meshes in $\mathbb{R}^d$ appears to be that of Rajan~\cite{rajan94optimality}, Waldron~\cite{waldron1998error}, Chen, Xu, and coworkers~\cite{chen2004optimal,chen2011efficient}, and Musin~\cite{musin1997properties}. In particular, the work of Waldron~\cite{waldron1998error} guarantees that standard Delaunay meshes minimize the pointwise, piecewise linear, interpolation error of functions with pointwise-bounded second derivatives. Unfortunately, this work is incomplete, as it does not apply to interpolation with high-order, piecewise polynomial functions. 

In the current paper, we prove a new set of results, which establish the quasi-optimality of gradient interpolation on protected Delaunay meshes in $\mathbb{R}^d$. Our results are more general than previous results, as they are applicable to interpolation with high-order, piecewise polynomial functions. 

Furthermore, we observe that (by definition) the gradient of a scalar function is a vector field. With this in mind, we extend our results which govern gradient interpolation in order to establish the quasi-optimality of vector-field interpolation on protected Delaunay meshes. 

\subsection{Paper Approach and Outline}

There are many possible approaches for analyzing the accuracy of high-order interpolation. In this work, we proceed in a straightforward fashion, and leverage vector calculus in conjunction with classical interpolation theory for polynomial functions. Briefly, our approach involves extending the notion of roughness, originally introduced by Rippa~\cite{rippa1990minimal} for gradients in $\mathbb{R}^2$, to the case of $\mathbb{R}^d$. Thereafter, we use this notion to obtain a set of results which govern the quasi-optimality of gradient interpolation and vector-field interpolation on protected Delaunay meshes. 

The format of the paper is as follows. In section~\ref{prelim_section}, we motivate the present work by introducing a canonical, elliptic problem in $\mathbb{R}^4$. In addition, we expand the definition of roughness, originally introduced by Rippa~\cite{rippa1990minimal}, into higher dimensions. In section~\ref{theory_section}, we present  the new theoretical results that govern the interpolation of gradients on meshes in $\mathbb{R}^d$. Here, we prove the quasi-optimality of protected Delaunay meshes for this purpose.  In section~\ref{vector_section}, we extend the results of section~\ref{theory_section}, in order to establish quasi-optimal interpolation properties for $L_2$-vector fields. Finally, section~\ref{conclusion_section} contains concluding remarks and suggestions for additional research.

\section{Preliminaries} \label{prelim_section}

In this section, we define some important notation. Thereafter, we introduce an example problem in $\mathbb{R}^4$. Next, we define the notion of roughness for this problem, and explain the relationship between roughness and the solution error. Lastly, we provide a precise definition for roughness in any number of dimensions.

\subsection{Notation}
Consider a simply-connected, polytopal domain $\Omega \subset \mathbb{R}^d$. On this domain, we can define the Sobolev space $L_2(\Omega)$ and its associated norm $\left\| \cdot \right\|_{L_{2}(\Omega)}$ as follows
\begin{align*}
    L_2(\Omega) &= \left\{f \, \Bigg| \, \int_{\Omega} f^2 \, dx_1 dx_2 \cdots d_{x_d} < \infty  \right\}, \qquad \left\| f \right\|_{L_{2}(\Omega)} &= \left[\int_{\Omega} f^2 \, dx_1 dx_2 \cdots d_{x_d} \right]^{1/2},
\end{align*}
where $f = f(\bm{x})= f(x_1, x_2, \ldots, x_d)$ is a scalar function. The vector-version of this space is associated with the following norm
\begin{align*}
    \left\| \bm{f} \right\|_{L_{2}(\Omega)} = \left[\int_{\Omega} \bm{f} \cdot \bm{f} \, dx_1 dx_2 \cdots d_{x_d}\right]^{1/2},   
\end{align*}
where $\bm{f} = \bm{f}(\bm{x}) = \bm{f}(x_1, x_2, \ldots, x_d) \in \mathbb{R}^{d}$ is a vector-valued function.

We can also define the following Sobolev spaces
\begin{align*}
    H^{1}(\Omega) &= \left\{f \in L_{2}(\Omega) \, |  D^{\bm{\alpha}} f \in L_{2}(\Omega), |\bm{\alpha}| \leq 1 \right\}, \qquad H^{1}_{0}(\Omega) &= \left\{f \in H^{1}(\Omega) \, | \, f = 0 \quad \mathrm{on} \quad \partial \Omega \right\},
\end{align*}
where
\begin{align*}
    D^{\bm{\alpha}} v &= \frac{\partial^{|\bm{\alpha}|}f}{\partial x^{\alpha_1} \partial x^{\alpha_2} \cdots \partial x^{\alpha_d}}, \qquad \bm{\alpha} = \left(\alpha_1, \alpha_2, \ldots, \alpha_d\right) \in  \mathbb{N}^d, \qquad |\bm{\alpha}| = \alpha_1 + \alpha_2 + \cdots + \alpha_d.
\end{align*}
Evidently, the associated norms are
\begin{align*}
    \left\| f \right\|_{H^{1}(\Omega)} &= \left[\int_{\Omega} \left( f^2 + \nabla f \cdot \nabla f \right) dx_1 dx_2 \cdots d_{x_d}  \right]^{1/2}, \quad \left\| f \right\|_{H^{1}_{0}(\Omega)} = \left[\int_{\Omega} \left( \nabla f \cdot \nabla f \right) dx_1 dx_2 \cdots d_{x_d}  \right]^{1/2},
\end{align*}
where $\nabla = \left(\frac{\partial}{\partial x_1}, \frac{\partial}{\partial x_2}, \ldots, \frac{\partial}{\partial x_d}  \right)$.

\subsection{Elliptic Problem in 4D} \label{elliptic_example}

Now, consider a simply-connected, polytopal domain $\Omega \subset \mathbb{R}^4$. We are interested in solving the following elliptic problem on $\Omega$
\begin{align}
    -\Delta u &= f, \qquad \text{in} \quad \Omega, \label{elliptic_strong} \\
   \nonumber u &= 0, \qquad \text{on} \quad \partial \Omega,
\end{align}
where $u$ is the twice-differentiable solution that vanishes on the boundary $\partial \Omega$, $f$ is a forcing function in $L_{2}(\Omega)$, and $\Delta = \nabla \cdot \nabla \left( \cdot \right)$ is the four-dimensional Laplacian. The four-dimensional gradient operator is given by $\nabla = \left(\frac{\partial}{\partial x}, \frac{\partial}{\partial y}, \frac{\partial}{\partial z}, \frac{\partial}{\partial w} \right)$.

We can formulate a finite element method for solving Eq.~\eqref{elliptic_strong} by replacing $u$ with $u_h$, multiplying by the test function $v_h$, integrating over the domain $\Omega$, and integrating by parts as follows
\begin{align*}
    &\int_{\Omega} \nabla u_h \cdot \nabla v_h \, dV - \int_{\partial \Omega} v_h \frac{\partial u_h}{\partial x} n_x \, dy dz dw - \int_{\partial \Omega} v_h \frac{\partial u_h}{\partial y} n_y \, dx dz dw \\[1.0ex]
    &- \int_{\partial \Omega} v_h \frac{\partial u_h}{\partial z} n_z \, dx dy dw - \int_{\partial \Omega} v_h \frac{\partial u_h}{\partial w} n_w \, dx dy dz = \int_{\Omega}  f v_h \, dV,
\end{align*}
where $u_h, v_h \in H^{1}(\Omega)$ and $dV = dx dy dz dw$. Next, boundary conditions can be enforced by choosing $u_h, v_h \in V_h \subset H^{1}_{0}(\Omega)$ so that
\begin{align*}
    \int_{\Omega} \nabla u_h \cdot \nabla v_h \, dV  = \int_{\Omega} f v_h \, dV.
\end{align*}
We can rewrite the expression above in terms of more familiar notation
\begin{align}
    a_h(u_h,v_h) = L_h(v_h), \label{elliptic_fem}
\end{align}
where 
\begin{align*}
    a_h(u_h,v_h) &\equiv \int_{\Omega} \nabla u_h \cdot \nabla v_h \, dV
    \\
    &= \int_{\Omega} \left( \frac{\partial u_h}{\partial x}  \frac{\partial v_h}{\partial x} + \frac{\partial u_h}{\partial y}  \frac{\partial v_h}{\partial y} + \frac{\partial u_h}{\partial z}  \frac{\partial v_h}{\partial z} + \frac{\partial u_h}{\partial w}  \frac{\partial v_h}{\partial w} \right) dx dy dz dw, \\[1.0ex]
    L_h(v_h) &\equiv \int_{\Omega} f v_h \, dV.
\end{align*}
It is well known that Eq.~\eqref{elliptic_fem} has a unique solution when $a_h$ is symmetric, bilinear, and governed by the following constraints
\begin{align*}
    | a_h(u_h,v_h) | \leq \sigma \left\| u_h \right\| \left\| v_h \right\|, \qquad \tau \left\| v_h \right\|^2 \leq a_h(v_h,v_h), 
\end{align*}
where $\sigma$ and $\tau$ are positive constants, and 
\begin{align*}
    \left\|v_h \right\| &\equiv \left[ \int_{\Omega} \left(v_{h}^{2} + \nabla v_h \cdot \nabla v_h \right) dV \right]^{1/2} \\[1.0ex]
    &= \left[ \int_{\Omega} \left(v_{h}^{2} + \left(\frac{\partial v_h}{\partial x} \right)^{2} + \left(\frac{\partial v_h}{\partial y} \right)^{2} + \left(\frac{\partial v_h}{\partial z} \right)^{2} + \left(\frac{\partial v_h}{\partial w} \right)^{2} \right) dV \right]^{1/2}.
\end{align*}
Next, we can define the energy norm
\begin{align*}
    \left\| v_h \right\|_{a} \equiv \sqrt{a_h(v_h,v_h)}.
\end{align*}
In accordance with standard elliptical theory (see~\cite{sayas2019variational}, Lemma 2.2), we can introduce the following energy functional
\begin{align}
    J(v_h) \equiv a_h(v_h,v_h) - 2 L_h(v_h). \label{jdef}
\end{align}
Here, $J(v_h)$ is the integral of the \emph{Lagrangian}. The minimum of this functional is the solution of Eq.~\eqref{elliptic_fem}. More precisely
\begin{align*}
    J(u_h) = \min_{v_h \in V_h} J(v_h).
\end{align*}
Now, we are ready to introduce the notion of roughness, and its relationship to solution error. In particular, suppose that we create a pair of meshes, $\mathcal{T}_1$ and $\mathcal{T}_2$, for our domain $\Omega$. These meshes are distinct, and do not necessarily possess the same number of simplices. Due to the bilinearity and symmetry of $a_h$, then the following equality holds
\begin{align}
    \left\|u - u_{h,1} \right\|^{2}_{a} = J(u_{h,1}) - J(u_{h,2}) + \left\|u - u_{h,2} \right\|^{2}_{a}.
\end{align}
Naturally, if we choose $\mathcal{T}_1$ such that
\begin{align*}
    J(u_{h,1}) \leq J(u_{h,2}),
\end{align*}
then it follows that
\begin{align*}
    \left\|u - u_{h,1} \right\|_{a} \leq \left\|u - u_{h,2} \right\|_{a}.
\end{align*}
Therefore, we seek to minimize $J(u_h)$ in order to minimize the error as measured by the energy norm $\left\| \cdot \right\|_{a}$. In turn, due to the definition of $J(u_h)$ (see Eq.~\eqref{jdef}), we seek meshes $\mathcal{T}_1$ that minimize the quantity $a_h(u_h,u_h)$. This quantity is often called the \emph{roughness} of the mesh. 
%Evidently, due to the roughness's direct correlation with the solution error, we seek to minimize it.

\subsection{Roughness in \texorpdfstring{$\mathbb{R}^d$}{Rd}}

It turns out that the measure of roughness that we identified in the previous section can be extended to any number of dimensions $d$. In particular, one may define
\begin{align}
    a(v,v) \equiv \int_{\Omega} \sum_{m=1}^{d} \left(\frac{\partial v}{\partial x_{m}}\right)^{2} \, dx_{1} dx_{2} \cdots dx_{d},
\end{align}
as the measure of roughness in $\mathbb{R}^d$. Here, we have omitted the subscript $h$ from the quantities $a$ and $v$ in order to simplify the notation. We note that the roughness functional (above) is merely the square of the norm on $H^{1}_{0}(\Omega)$, introduced previously. 

\section{Theoretical Results: Gradient Interpolation} \label{theory_section}

 \subsection{Mesh Properties}

Let us introduce a finite set of points $S$ 
\begin{align*}
    S = \left\{\bm{p}_{1}, \bm{p}_{2}, \ldots, \bm{p}_{l}, \ldots, \bm{p}_{N_{v}} \right\}, \qquad 1\leq  l \leq N_v,
\end{align*}
where the cardinality $|S| = N_v$, and each point $\bm{p}_{l}$ is contained within a bounded, simply-connected domain $\Omega \in \mathbb{R}^{d}$. For the sake of simplicity, we let $\mathrm{conv}(S) = \Omega$. Furthermore, we assume that $S$ is a $(\varepsilon, \overline{\eta})$-net, which satisfies the following conditions:
\begin{align}
    \forall \bm{x} \in \Omega, \exists \bm{p}_{l} \in S&: \qquad \left| \bm{x} - \bm{p}_{l} \right| \leq \varepsilon, \label{density_condition} \\[1.0ex]
    \forall \bm{p}_{l}, \bm{p}_{n} \in S&: \qquad \left| \bm{p}_{l} - \bm{p}_{n} \right| \geq \eta, \label{separation_condition}
\end{align}
where $\varepsilon > 0$, $\eta > 0$, and $\overline{\eta} = \eta/\varepsilon$. Together, Eqs.~\eqref{density_condition} and \eqref{separation_condition} control the density of the points.

Now, let us introduce a generic mesh $\mathcal{T}$ whose vertices are the points of $S$, and whose elements $K$ are non-overlapping.  We assume that the mesh is a triangulation, and that the union of all elements has the same hypervolume as the domain itself. Furthermore, let us assume that each $K$ is a $d$-simplex equipped with $(d+1)$-facets of dimension $d-1$. In addition, we assume that each facet is either a boundary facet (with precisely one $d$-simplex neighbor) or an interior facet (with precisely two $d$-simplex neighbors). As a result, the mesh does not contain any hanging nodes. In accordance with these assumptions (above), we say that the mesh is a \emph{pseudo manifold}: i.e.~a pure simplicial $d$-complex which is $d$-connected, and for which each $(d-1)$-simplex has exactly one or two $d$-simplex neighbors. 

The $d+1$ vertices of each simplex $K$ are denoted by $\bm{p}_{K,i}$, where $i = 1, \ldots, d+1$. The coordinates of each vertex are given by
\begin{align}
    \bm{p}_{K,i} = \left(p_{K,i}^{1}, \, p_{K,i}^{2}, \ldots, \, p_{K,i}^{m}, \ldots, \, p_{K,i}^{d}\right)^{T},
\end{align}
where $1 \leq m \leq d$. The absolute value of a vertex can be defined in terms of the absolute values of its components, as follows
\begin{align}
    \mathrm{abs}(\bm{p}_{K,i}) = \left(\left|p_{K,i}^{1}\right|, \, \left|p_{K,i}^{2}\right|, \ldots, \, \left|p_{K,i}^{m}\right|, \ldots, \, \left|p_{K,i}^{d}\right|\right)^{T}.
\end{align}
Note that $|\bm{p}_{K,i}|$ and $\mathrm{abs}(\bm{p}_{K,i})$ are not equivalent quantities, as one is the scalar magnitude of a $d$-vector, and the other is a $d$-vector of absolute values. 

Lastly, we can define the $d(d+1)/2$ edges of each simplex $K$ as follows
\begin{align}
    \bm{p}_{K,ij} &= \bm{p}_{K,j} - \bm{p}_{K,i} =\left(p_{K,ij}^{1}, \, p_{K,ij}^{2}, \ldots, \, p_{K,ij}^{m}, \ldots, \, p_{K,ij}^{d}\right)^{T},
\end{align}
where $1\leq i \leq d+1$, $1\leq j \leq i-1$, and $1 \leq m \leq d$.

\subsection{The Sizing Function}

Next, we introduce a sizing function $\mathcal{D}(\bm{x}) = \mathcal{D}(x_1, x_2, \ldots, x_d) >0$. We assume that this sizing function has units of length, and that the squared-reciprocal of this function is $C^2$-continuous over the domain~$\Omega$, i.e.
\begin{align*}
    \frac{1}{\mathcal{D}^2(\bm{x})} \in C^{2}(\Omega).
\end{align*}
%
% Evidently, the value of this function at each vertex $\bm{p}_{l}$ is given by
% %
% \begin{align*}
%     \frac{1}{\mathcal{D}^2(\bm{p}_{l})}.
% \end{align*}
%
We can relate the sizing function to the local mesh spacing, as follows
\begin{align}
   \label{size_constraint_one} &\min\left[\frac{1}{\mathcal{D}^{2}(\bm{p}_{K,1})}, \frac{1}{\mathcal{D}^{2}(\bm{p}_{K,2})}, \ldots, \frac{1}{\mathcal{D}^{2}(\bm{p}_{K,d+1})} \right] \leq \left(\frac{1}{\Delta(K)^2} + \zeta_{K}\right), \\[1.0ex] 
    \nonumber & \max\left[\frac{1}{\mathcal{D}^{2}(\bm{p}_{K,1})}, \frac{1}{\mathcal{D}^{2}(\bm{p}_{K,2})}, \ldots, \frac{1}{\mathcal{D}^{2}(\bm{p}_{K,d+1})} \right] \geq \left(\frac{1}{\Delta(K)^2} + \zeta_{K}\right),
\end{align}
where $\Delta(K)$ is the diameter of $K$ (the longest edge length of $K$), and where we choose $\zeta_{K} \in \mathbb{R}$ so that the inequalities above hold true. Let us interpret Eq.~\eqref{size_constraint_one} in more intuitive terms. Suppose that Eq.~\eqref{size_constraint_one} holds with $\zeta_{K} = 0$. This means that $1/\Delta(K)^2$ is bounded above and below by the minimum and maximum values of $1/\mathcal{D}^{2}(\bm{x})$ evaluated at the vertices of $K$. This indicates that the size of each element $K$ (as characterized by $\Delta(K)$) conforms to the specifications of the sizing field, $\mathcal{D}(\bm{x})$. Conversely, if $\zeta_K$ is non-zero, then $1/\Delta(K)^2$ is \emph{not} bounded above and below by the minimum and maximum values of $1/\mathcal{D}^{2}(\bm{x})$ evaluated at the vertices of $K$. In this sense, $\zeta_{K}$ characterizes the difference between the \emph{actual} mesh spacing and the \emph{suggested} mesh spacing provided by the sizing function. 

We may now construct a linear interpolant $\mathcal{L}_{[1/\mathcal{D}^2]}(\bm{x})$ which takes on the values of $1/\mathcal{D}^{2}(\bm{x})$ at the mesh vertices $\bm{p}_{l}$, such that
\begin{align}
    \mathcal{L}_{[1/\mathcal{D}^2]} (\bm{p}_{l}) = \frac{1}{\mathcal{D}^{2}(\bm{p}_{l})}, \qquad \forall l.
\end{align}
As a consequence of Eq.~\eqref{size_constraint_one} and the linearity of the interpolant, there exists at least one point $\bm{x}_{K} \in K$ such that
\begin{align}
    \mathcal{L}_{[1/\mathcal{D}^2]} (\bm{x}_{K}) = \frac{1}{\Delta(K)^{2}} + \zeta_{K},
    \label{size_constraint_two}
\end{align}
for each $K$. For the sake of additional clarity, we have included a figure, (Figure~\ref{fig:sizing}), which summarizes the role of the sizing function for a tetrahedron.
\begin{figure}[h!]
    \centering
    \includegraphics[width = 1.1\textwidth]{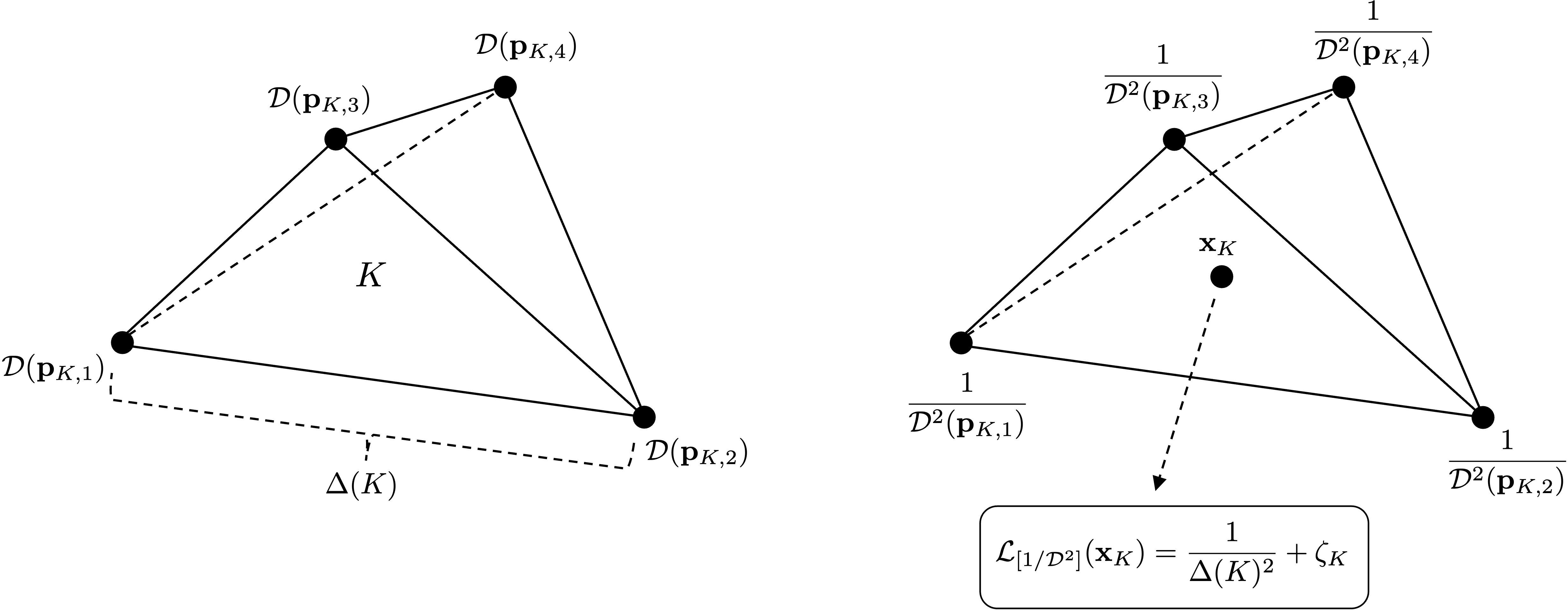}
    \caption{Left, the sizing function for a tetrahedron element $K$ evaluated at its vertices, $\bm{p}_{K,i}$. Here, the longest edge length between vertices is denoted by $\Delta(K)$. Right, the squared-reciprocal of the sizing function evaluated at the vertices, and its linear interpolant $\mathcal{L}_{[1/\mathcal{D}^2]}(\bm{x})$.}
    \label{fig:sizing}
\end{figure}

We are now ready to proceed with our analysis.

\subsection{Derivation of Results} \label{isotropic_theory_section}

In this section, our objective is to construct an error estimate for gradient interpolation on the mesh $\mathcal{T}$. Towards this end, we introduce  definitions for a \emph{roughness} functional and a gradient norm. Next, we prove that the roughness function is equivalent to the gradient norm. Thereafter, we construct a series of upper bounds for the roughness functional and the gradient norm, and then leverage these results to construct an upper bound for the error in the gradient norm. 

\begin{definition}[Roughness Functional]
Consider the following, non-negative functional of $v$ over the mesh $\mathcal{T}$
\begin{align}
    \Psi_{\mathcal{T}}\left(\nabla v\right) \equiv  \left[\sum_{K \in \mathcal{T}} \frac{1}{h_{K}^{2}} \int_{K} \sum_{i=1}^{d+1} \sum_{j=1}^{i-1}\left(\mathrm{abs}(\nabla v)\cdot \mathrm{abs}(\bm{p}_{K,ij}) \right)^{2} dV \right]^{1/2},
\end{align}
where we assume that $v$ is a piecewise-$H^1$-vector field on $\Omega$, and $h_{K}$ is a characteristic length scale associated with each $K$. The presence of the $\mathrm{abs}$ symbols above makes it easier for us to construct a lower bound for $\Psi_{\mathcal{T}}\left(\nabla v\right)$ in the work that follows.
\label{isotropic_functional}
\end{definition}

%We can also define the gradient norm in a similar fashion.
%
\begin{definition}[Gradient Norm]
Consider the following gradient norm of $v$ over the mesh $\mathcal{T}$
\begin{align}
    \left\| \nabla v \right\|_{L_{2}(\Omega)} &\equiv \left[ \sum_{K\in \mathcal{T}} \int_{K} \nabla v \cdot \nabla v \, dV \right]^{1/2},
    %= \left[ \sum_{K\in \mathcal{T}} \int_{K} \sum_{\ell=1}^{d} \left(\nabla v \cdot \bm{e}_{\ell} \right)^2 dV \right]^{1/2},
\end{align}
where we assume that $v$ is a piecewise-$H^1$-vector field on $\Omega$.
\label{isotropic_gradient_norm}
\end{definition}

We are now ready to show that the roughness functional and the gradient norm are equivalent.

\begin{theorem}[Equivalence of the Roughness Functional and the Gradient Norm]
    The functional in Definition~\ref{isotropic_functional} and the gradient norm in Definition~\ref{isotropic_gradient_norm} are equivalent in the following sense
    \begin{align}
        C_1 \left\| \nabla v \right\|_{L_{2}(\Omega)} \leq \Psi_{\mathcal{T}}\left(\nabla v\right) \leq  C_2 \left\| \nabla v \right\|_{L_{2}(\Omega)},
        \label{functional_equivalence}
    \end{align}
    where $v$ resides in the space of piecewise-$H^1$-vector fields on $\Omega$, and $C_1$ and $C_2$ are constants that depend on the mesh $\mathcal{T}$,
    \begin{align}
        C_{1} &= \sqrt{\frac{d+1}{2d}} \min_{K} \left( \frac{\min_{s} \left[\mathrm{dist}\left(\bm{p}_{K,s}, \mathrm{aff}(\mathcal{F}_{K,s})\right)\right]}{\Delta(K)} \right), \label{sliver_constant} \\[1.0ex]
        C_{2} &= \max_{K} \left( \frac{1}{\Delta(K)} \sqrt{\sum_{i=1}^{d+1} \sum_{j=1}^{i-1}\left(\bm{p}_{K,ij}\cdot \bm{p}_{K,ij} \right) }\right). \label{coarse_constant}
    \end{align}
    Here, $\mathcal{F}_{K,s}$ is the facet opposite to the vertex $\bm{p}_{K,s}$ and $1 \leq s \leq d+1$.
    \label{functional_equivalence_lemma}
\end{theorem}

\begin{proof}
    See Appendix~\ref{functional_equivalence_lemma_proof}.
\end{proof}

The mesh-dependent constants $C_1$ (Eq.~\eqref{sliver_constant}) and $C_2$ (Eq.~\eqref{coarse_constant}) appear to be new. In what follows, we will interpret their meanings.

\begin{remark}[Sliver-Detecting Constant]
     In order to interpret the constant $C_1$, it is helpful for us to first define the concept of \emph{thickness}. In accordance with~\cite{boissonnat2012stability}, the thickness of a $d$-simplex $K$ is given by
     \begin{align}
         \Xi(K) = \begin{cases}
             1 & \text{if} \; d = 0 \\
             \min_{s} \left( \frac{\mathrm{dist}\left(\bm{p}_{K,s}, \mathrm{aff}(\mathcal{F}_{K,s})\right)}{d \, \Delta(K)} \right) & \text{otherwise}.
         \end{cases}
     \end{align}
     Evidently, the minimum thickness, $\min_{K} (\Xi(K))$, will be small for meshes which contain sliver elements, and will be large in the absence of slivers.  We can rewrite $C_1$ in terms of $\Xi(K)$ as follows
     \begin{align}
         \label{sliver_constant_bound} C_1 &= \sqrt{\frac{d+1}{2d}} \min_{K} \left( \min_{s} \left[ \frac{\mathrm{dist}\left(\bm{p}_{K,s}, \mathrm{aff}(\mathcal{F}_{K,s})\right)}{d \, \Delta(K)} \right] d \right) = \sqrt{\frac{d(d+1)}{2}} \min_{K} \Xi(K).
     \end{align}
     Therefore, the constant $C_1$ is a direct measure of the minimum element thickness.
     \label{sliver_constant_remark}
\end{remark}

\begin{remark}[Protected Delaunay Meshes and Thickness]
     Broadly speaking, the minimum element thickness is difficult to control in $\mathbb{R}^d$, for arbitrary $d$. Fortunately, in Lemma 5.27 of~\cite{boissonnat2018geometric}, Boissonnat and coworkers proved that the thicknesses of simplices are bounded below as follows
     \begin{align}
         \Xi(K) \geq \frac{\delta^2}{8 d \varepsilon^2}. \label{thickness_bound}
     \end{align}
     Therefore, the element thickness depends quadratically on the protection, $\delta$.
     It is important to note that, with the exception of Eq.~\eqref{thickness_bound}, we are unaware of any results which provide rigorous mathematical guarantees for the thickness of simplices in an arbitrary number of dimensions. 

     Upon combining Eq.~\eqref{thickness_bound} with Eq.~\eqref{sliver_constant_bound}, we obtain
     \begin{align}
         C_1 \geq \sqrt{\frac{d+1}{2d}} 
         \left(\frac{\delta^2}{8 \varepsilon^2}\right).
     \end{align}
     Therefore, direct control of $C_1$ can be obtained on protected Delaunay meshes.
     \label{protected_delaunay_remark}
\end{remark}

\begin{remark}[Maximum Protection]
     In light of the results above, we seek a procedure for generating protected Delaunay meshes with a maximal level of protection $\delta$ for arbitrary $d$. The current best procedure for generating \emph{unstructured}, protected Delaunay meshes is given by~\cite{boissonnat2018geometric}. Here, the protection is guaranteed to be at least
     \begin{align}
         \delta \sim \mathcal{O}\left(\frac{1}{2^{d^2}}\right).
         \label{unstructured_protection}
     \end{align}
     In addition, the current best procedure for generating \emph{structured} protected Delaunay meshes is given by \emph{Coxeter reflection}~\cite{choudhary2020coxeter}. In particular, the Coxeter triangulations of type $\widetilde{A}_d$ are protected Delaunay triangulations, which are guaranteed to have the following protection
     \begin{align}
         \delta \sim \mathcal{O}\left(\frac{1}{d^2}\right).
         \label{structured_protection}
     \end{align}
     This level of protection is the largest of any known triangulation. In this sense, we argue that structured protected Delaunay meshes are best for maximizing the protection in $\mathbb{R}^d$, and thereby, maximizing the constant $C_1$.

     Finally, we note that the gap between the protection estimates for the unstructured and structured cases (Eqs.~\eqref{unstructured_protection} and \eqref{structured_protection}, respectively) can likely be improved. In particular, the protection estimate for the unstructured case decays exponentially with the number of dimensions $d$, and is therefore, an example of the `curse of dimensionality'. However, given the exceptional quality of the structured estimate, we are optimistic that the unstructured estimate can be improved. 
\end{remark}

\begin{remark}[Max Edge-Length Constant]
    The constant $C_2$ is easier to interpret than $C_1$. By inspection, $C_2$ is merely the square root of the maximum of the sum of squared edge lengths, normalized by $\Delta(K)$, and evaluated across all simplices in the mesh $\mathcal{T}$. A more precise interpretation of this constant is unnecessary in the context of the present work, as it does not appear in any of the subsequent estimates.   
\end{remark}
    
In what follows, we will introduce additional upper bounds for both the roughness functional and the gradient norm.

\begin{lemma}[Upper Bound for the Roughness Functional]
    The functional in Definition~\ref{isotropic_functional} is bounded above by the infinity norm of the gradient as follows
    \begin{align}
        \label{iso_upper_bound}
         \Psi_{\mathcal{T}}\left(\nabla v\right) \leq C_{3} \sqrt{\Theta} \left\| \nabla v \right\|_{L_{\infty}(\Omega)},
    \end{align}
    where $v$ resides in the intersection of $L_{\infty}$-vector fields and piecewise-$H^1$-vector fields over $\Omega$, $\Theta$ is a functional that depends on the mesh $\mathcal{T}$, and $C_3$ is a constant that also depends on the mesh 
    \begin{align}
        C_{3} &= \sqrt{\min \left[\frac{1}{2} \max_{K} \left( R_{K,\mathrm{min}}\right)^{2}  \left\| \frac{1}{\mathcal{D}^{2}} \right\|_{2,L_{\infty}(\Omega)} + \left\| \frac{1}{\mathcal{D}^{2}} \right\|_{L_{\infty}(\Omega)} + \max_{K} \left| \zeta_{K} \right| , \frac{1}{\eta^2} \right]}, \\[1.0ex]
        \Theta &= \sum_{K \in \mathcal{T}} \left(\sum_{i=1}^{d+1} \sum_{j=1}^{i-1}\left(\bm{p}_{K,ij}\cdot \bm{p}_{K,ij} \right) |K| \right).
    \end{align}
    Here, $R_{K,\min}$ is the min-containment radius.
    \label{iso_upper_bound_lemma}
\end{lemma}

\begin{proof}
    See Appendix~\ref{iso_upper_bound_lemma_proof}.
\end{proof}

Let us now interpret the new quantities $C_3$ and $\Theta$.

\begin{remark}[Min-Containment Radius Constant]
    The behavior of $C_3$ is directly influenced by the largest min-containment radius of an element in the mesh $\mathcal{T}$, the smoothness of the sizing function $1/\mathcal{D}^{2}(\bm{x})$, the difference between the local element size and the sizing function (given by $\zeta_{K}$), and the separation of points (given by $\eta)$. Suppose that we specify a particular function $1/\mathcal{D}^{2}(\bm{x})$ $\emph{a priori}$, and the discrepancy in the element size $\zeta_{K}$ is precisely controlled by the distribution of mesh nodes $\bm{p}_{l}$, $l = 1, \ldots, N_v$, which is also known \emph{a priori}. Under these circumstances, the mesh closely conforms to the sizing function, and the magnitude of the constant $C_3$ only depends on the min-containment radius. Furthermore, in this case $C_3$ is minimized on a standard Delaunay mesh in $\mathbb{R}^d$, as Rajan~\cite{rajan94optimality} proved that such a mesh minimizes the maximum min-containment radius, relative to any other mesh that uses the same set of points. 
    \label{min_containment_remark}
\end{remark}

\begin{remark}[Rajan's Functional]
    The quantity $\Theta$ is a functional of the mesh $\mathcal{T}$. It is the `weighted-sum of edge lengths squared', where the weighting quantity is the $d$-hypervolume of each simplex $K$. This functional was originally identified by Rajan~\cite{rajan94optimality}. His version of the functional included an extra factor of $1/(d+1)(d+2)$, such that
    \begin{align}
        \widehat{\Theta} = \frac{\Theta}{(d+1)(d+2)}  = \frac{1}{(d+1)(d+2)} \sum_{K \in \mathcal{T}} \left(\sum_{i=1}^{d+1} \sum_{j=1}^{i-1}\left(\bm{p}_{K,ij}\cdot \bm{p}_{K,ij} \right) |K| \right).
    \end{align}
    In~\cite{rajan94optimality}, Rajan proves that a standard Delaunay mesh in $\mathbb{R}^{d}$ is guaranteed to minimize $\widehat{\Theta}$ relative to any other mesh which uses the same set of points. Evidently, this Delaunay-minimization property immediately extends to $\Theta$. 
    \label{rajan_functional_remark}
\end{remark}

\begin{corollary}[Upper Bound for the Gradient Norm]
    The norm of the gradient in Definition~\ref{isotropic_gradient_norm} is bounded above by the infinity norm of the gradient as follows
    \begin{align}
         \left\| \nabla v \right\|_{L_{2}(\Omega)} \leq  \frac{C_3 \sqrt{\Theta}}{C_1} \left\| \nabla v \right\|_{L_{\infty}(\Omega)},
    \end{align}
    where $\Theta$ is a functional that depends on the mesh $\mathcal{T}$, and $C_1$ and $C_3$ are constants that depend on the mesh. In addition, $v$ resides in the intersection of $L_{\infty}$-vector fields and piecewise-$H^1$-vector fields over $\Omega$.
    \label{lemma_upper_bound_isotropic_gradient_norm}
\end{corollary}

\begin{proof}
    The result follows immediately from combining Eq.~\eqref{functional_equivalence} from Theorem~\ref{functional_equivalence_lemma} with Eq.~\eqref{iso_upper_bound} from Lemma~\ref{iso_upper_bound_lemma}.
\end{proof}

Next, we introduce some results which govern the interpolation of the gradient on our canonical mesh $\mathcal{T}$. These results leverage classical interpolation theory, and hold true for polynomials of any degree.

\begin{definition}[Gradient Interpolation]
    Consider the following interpolation of the gradient on element $K \in \mathcal{T}$
    \begin{align}
       (\nabla v)_{h,i} &= \left(\frac{\partial v}{\partial x_i} \right)_h, \qquad \text{where} \qquad \left(\frac{\partial v}{\partial x_i} \right)_h = \sum_{j=1}^{N_p} \frac{\partial v}{\partial x_i}(\bm{x}_j)L_j(\bm{\xi}(\bm{x})),
    \end{align}
    and where $i = 1, \ldots, d$ are indexes, $\bm{x}$ is a vector of coordinates in physical space $(x_1, x_2, \ldots, x_d)$, $\bm{\xi}$ is a vector of coordinates in reference space $(\xi_1, \xi_2, \ldots, \xi_{d})$, $\bm{\xi}(\bm{x})$ is the map from physical space to reference space, $\bm{x}(\bm{\xi})$ is the inverse map from reference space to physical space, $\bm{x}_j = \bm{x}(\bm{\xi}_j)$ is an interpolation point on the element $K$, $N_p$ is the total number of interpolation points, and $L_j(\bm{\xi})$ is a multi-dimensional Lagrange polynomial which assumes the value of one at $\bm{\xi}_j$ and zero at all other interpolation points. One may consult Figure~\ref{element_map_fig} for an illustration of the relationship between the physical and reference space elements in three dimensions.

    Finally, an explicit formula for $N_p$ is
    \begin{align}
        N_p = \frac{(k+d)!}{k!d!},
    \end{align}
    for polynomials of degree $\leq k$.
\end{definition}

\begin{figure}[h!]
    \centering
    \includegraphics[width = 1.0\textwidth]{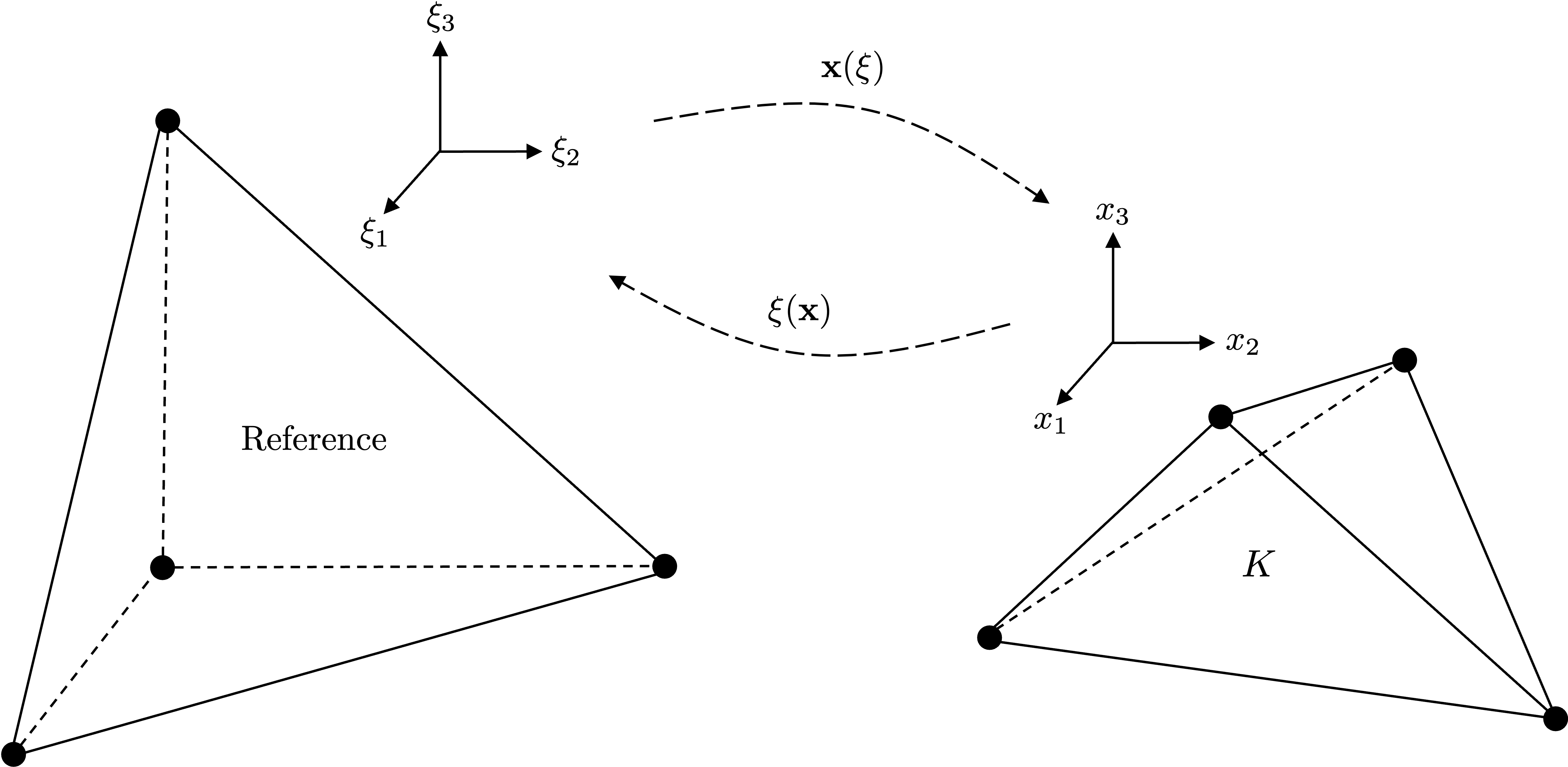}
    \caption{Mapping between the reference element and the physical element $K$ in three dimensions.}
    \label{element_map_fig}
\end{figure}

We are now ready to introduce a theorem which governs the error of gradient interpolation.

\begin{theorem}[Error Estimate for Gradient Interpolation]
    A measure of the error between the gradient $\nabla v$ and a piecewise polynomial interpolation of the gradient $(\nabla v)_h$ on the mesh $\mathcal{T}$ is given by
    \begin{align}
         \Psi_{\mathcal{T}}\left(\nabla v - (\nabla v)_h\right) \leq (1 + \Lambda) C_{3} \sqrt{\Theta} \left\| \nabla v - (\nabla v)^{\ast} \right\|_{L_{\infty}(\Omega)},
         \label{functional_error_estimate}
    \end{align}
    where $(\nabla v)^{\ast}$ is the piecewise polynomial `best approximation' of the gradient over $\mathcal{T}$, $\Lambda$ is the Lebesgue constant, $\Theta$ is a functional of the mesh, and $C_3$ is a constant that also depends on the mesh.
    \label{functional_error_estimate_lemma}
\end{theorem}

\begin{proof}
    See Appendix~\ref{functional_error_estimate_lemma_proof}.
\end{proof}

\begin{corollary}[Simplified Error Estimate for Gradient Interpolation]
     A measure of the error between the gradient $\nabla v$ and a piecewise polynomial interpolation of the gradient $(\nabla v)_h$ on the mesh $\mathcal{T}$ is given by
    \begin{align}
    \left\| \nabla v - (\nabla v)_h \right\|_{L_{2}(\Omega)} \leq \frac{(1 + \Lambda) C_{3} \sqrt{\Theta}}{C_1} \left\| \nabla v - (\nabla v)^{\ast} \right\|_{L_{\infty}(\Omega)},
         \label{norm_error_estimate}
    \end{align}
    where $(\nabla v)^{\ast}$ is the piecewise polynomial `best approximation' of the gradient over $\mathcal{T}$, $\Lambda$ is the Lebesgue constant, $\Theta$ is a functional that depends on the mesh, and $C_1$ and $C_3$ are constants that depend on the mesh.
    \label{modified_corollary}
\end{corollary}

\begin{proof}
    We begin by replacing $\nabla v$ with $(\nabla v - (\nabla v)_h)$ in Eq.~\eqref{functional_equivalence} from Theorem~\ref{functional_equivalence_lemma}, such that
    \begin{align*}
        C_1 \left\| \nabla v - (\nabla v)_h \right\|_{L_{2}(\Omega)} \leq \Psi_{\mathcal{T}}\left(\nabla v - (\nabla v)_h\right) \leq  C_2 \left\| \nabla v - (\nabla v)_h \right\|_{L_{2}(\Omega)}.
    \end{align*}
    Thereafter, combining the expression above with Eq.~\eqref{functional_error_estimate} from Theorem~\ref{functional_error_estimate_lemma} yields the desired result.
\end{proof}

\begin{remark}[Error Estimate Interpretation] \label{optimal_delaunay_remark}
    Corollary~\ref{modified_corollary} is important, because it characterizes the dominant sources of error in high-order, piecewise polynomial gradient interpolation. Broadly speaking, the interpolation error is \emph{amplified} by four factors:
    \begin{itemize}
        \item Small values of the constant $C_1$. This constant becomes small when the mesh contains sliver elements, (see Remark~\ref{sliver_constant_remark}). Recall that the element thickness can become arbitrarily small on a standard Delaunay mesh for $d>2$. However, on a protected Delaunay mesh, we are guaranteed a lower bound for the minimum thickness of our elements. This fact was established in Remark~\ref{protected_delaunay_remark}.
        
        \item Large values of the constant $C_3$. This constant becomes large when the maximum min-containment radius becomes large. The maximum min-containment radius is minimized on a standard Delaunay mesh. This fact was established by Rajan in~\cite{rajan94optimality}, and discussed in Remark~\ref{min_containment_remark}.
        
        \item Large values of the mesh functional $\Theta$. This functional can be minimized by using a standard Delaunay mesh. This fact was established by Rajan in~\cite{rajan94optimality}, and discussed in Remark~\ref{rajan_functional_remark}.
        
        \item Large values of the Lebesgue constant $\Lambda$. This constant can be minimized by carefully choosing the locations of the interpolation points within each element. The process of finding Lebesgue-optimized interpolation points is a thriving industry, and we refer the interested reader to~\cite{warburton2006explicit,isaac2020recursive,gobel2024explicit} for detailed discussions of this topic. 
    \end{itemize}
    Based on the discussion above, one may seek to minimize gradient interpolation error by using Lebesgue-optimized interpolation points on a protected Delaunay mesh. 
\end{remark}

\section{Theoretical Results: Vector-Field Interpolation} \label{vector_section}

Let us briefly shift our attention to the problem of interpolating a vector-valued function $\bm{f} = \bm{f}(\bm{x}) = \bm{f}(x_1, x_2, \ldots, x_d)$ defined on the domain $\Omega$, where $\bm{f} \in \left[L_{2}(\Omega)\right]^{d}$. In this case, we can define the following function:

\begin{definition}[Edge Functional]
Consider the following, non-negative functional of $\bm{f}$ over the mesh~$\mathcal{T}$
\begin{align}
    \Psi_{\mathcal{T}}\left(\bm{f} \right) \equiv  \left[\sum_{K \in \mathcal{T}} \frac{1}{h_{K}^{2}} \int_{K} \sum_{i=1}^{d+1} \sum_{j=1}^{i-1}\left(\mathrm{abs}(\bm{f})\cdot \mathrm{abs}(\bm{p}_{K,ij}) \right)^{2} dV \right]^{1/2},
\end{align}
where we assume that $\bm{f}$ is an $L_2$-vector field on $\Omega$, and $h_K$ is a characteristic length scale associated with each $K$. 
\label{edge_functional}
\end{definition}

% \begin{definition}[$L_2$-Vector Norm]
% Consider the following norm of $\bm{f}$ over the mesh~$\mathcal{T}$
% %
% \begin{align}
%     \left\| \bm{f} \right\|_{L_{2}(\Omega)} &\equiv \left[ \sum_{K\in \mathcal{T}} \int_{K} \bm{f} \cdot \bm{f} \, dV \right]^{1/2},
% \end{align}
% %
% where we assume that $\bm{f}$ is a $L_2$-vector field on $\Omega$.
% \label{vector_norm}
% \end{definition}

With this definition in mind, we note that Theorem~\ref{functional_equivalence_lemma}, Lemma~\ref{iso_upper_bound_lemma}, Corollary~\ref{lemma_upper_bound_isotropic_gradient_norm}, and Theorem~\ref{functional_error_estimate_lemma} hold with $\bm{f}$ in place of $\nabla v$. In addition, the following corollary holds.
\begin{corollary}[Simplified Error Estimate for Vector Interpolation]
     A measure of the error between the vector, $\bm{f}$, and a piecewise polynomial interpolation of the vector, $\bm{f}_h$, on the mesh $\mathcal{T}$ is given by
    \begin{align}
    \left\| \bm{f} - \bm{f}_h \right\|_{L_{2}(\Omega)} \leq \frac{(1 + \Lambda) C_{3} \sqrt{\Theta}}{C_1} \left\| \bm{f} - \bm{f}^{\ast} \right\|_{L_{\infty}(\Omega)},
         \label{vector_norm_error_estimate}
    \end{align}
    where $\bm{f}^{\ast}$ is the piecewise polynomial `best approximation' of the vector $\bm{f}$ over $\mathcal{T}$, $\Lambda$ is the Lebesgue constant, $\Theta$ is a functional that depends on the mesh, and $C_1$ and $C_3$ are constants that depend on the mesh.
    \label{vector_modified_corollary}
\end{corollary}

\begin{proof}
    The proof follows immediately from the proof of Corollary~\ref{modified_corollary}, with $\bm{f}$ in place of $\nabla v$.
\end{proof}

\begin{remark}
    Recall that Remark~\ref{optimal_delaunay_remark} describes the conditions under which we can obtain quasi-optimal interpolation of gradients $\nabla v$. These conditions are also required in order to obtain quasi-optimal interpolation of $L_2$-vector fields $\bm{f}$.
\end{remark}

\section{Conclusion} \label{conclusion_section}

In this work, we extend the concept of roughness, originally introduced by Rippa~\cite{rippa1990minimal}, into dimensions greater than two. We construct a roughness functional, which is inspired by the variational formulations of elliptic problems in $\mathbb{R}^d$. We prove that this roughness functional is equivalent to the $L_2$-norm of the gradient. In addition, we leverage this equivalence in order to construct an upper bound for the $L_2$-norm of the gradient interpolation error. This upper bound is written in terms of the best-approximation error in the infinity norm, Rajan's functional, the Lebesgue constant, and two new constants: a \emph{sliver-detecting} constant, and a \emph{min-containment radius} constant. Based on these results, we conclude that quasi-optimal, high-order, piecewise polynomial gradient interpolation can be achieved on a protected Delaunay mesh, equipped with Lebesgue-optimized interpolation points. Finally, we show that our analysis of quasi-optimal gradient interpolation extends to quasi-optimal vector-field interpolation. 

To our knowledge, we have constructed the first mathematical results which govern the optimality of high-order, piecewise polynomial interpolation on protected Delaunay meshes. In addition, we have identified one of the few practical applications of Rajan's optimality results~\cite{rajan94optimality}. Based on this work, we believe there are now stronger incentives to develop additional  optimality results which govern geometric functionals.

%%
%% Bibliography
%%

%% Please use bibtex, 

\bibliography{references}

\appendix

\section{Proofs} \label{first_appendix}

\subsection{Proof of Theorem~\ref{functional_equivalence_lemma}}
\label{functional_equivalence_lemma_proof}

\begin{proof}
    We start by establishing the lower bound and identifying $C_1$. From the definition of $\Psi_{\mathcal{T}}\left(\nabla v\right)$, we have
\begin{align}
    \label{iso_lower_two}
    \left(\Psi_{\mathcal{T}}\left(\nabla v\right)\right)^{2} &= \sum_{K \in \mathcal{T}} \frac{1}{h_{K}^{2}} \int_{K} \sum_{i=1}^{d+1} \sum_{j=1}^{i-1}\left(\mathrm{abs}(\nabla v)\cdot \mathrm{abs}(\bm{p}_{K,ij}) \right)^{2} dV
    \\[1.0ex]
    \nonumber &\geq \sum_{K \in \mathcal{T}} \frac{1}{h_{K}^{2}} \int_{K} \min_{m} \left(\sum_{i=1}^{d+1} \sum_{j=1}^{i-1} \left( p_{K,ij}^{m} \right)^2 \right) \left( \nabla v\cdot \nabla v \right) dV,
\end{align}
where we have used the following sequence of inequalities
\begin{align}
    \label{expanded_product}
   \sum_{i=1}^{d+1} \sum_{j=1}^{i-1} \left(\mathrm{abs}(\nabla v)\cdot \mathrm{abs}(\bm{p}_{K,ij}) \right)^{2} &= \sum_{i=1}^{d+1} \sum_{j=1}^{i-1} \left( \sum_{m=1}^{d} |(\nabla v)^{m}| |p_{K,ij}^{m}| \right)^{2} \\[1.0ex]
   \nonumber &\geq \sum_{i=1}^{d+1} \sum_{j=1}^{i-1} \left[ \sum_{m=1}^{d} \left( |(\nabla v)^{m}| |p_{K,ij}^{m}| \right)^{2} \right] \\[1.0ex]
   \nonumber &= \sum_{i=1}^{d+1} \sum_{j=1}^{i-1} \left[ \sum_{m=1}^{d} \left( (\nabla v)^{m} \right)^{2} \left(p_{K,ij}^{m} \right)^{2} \right] \\[1.0ex]
    \nonumber &=  \sum_{m=1}^{d} \left[ \left( (\nabla v)^{m} \right)^{2} \sum_{i=1}^{d+1} \sum_{j=1}^{i-1} \left(p_{K,ij}^{m} \right)^{2}\right] \\[1.0ex]
   \nonumber &\geq \min_{m} \left(  \sum_{i=1}^{d+1} \sum_{j=1}^{i-1} \left(p_{K,ij}^{m} \right)^{2} \right) \sum_{m=1}^{d} \left( (\nabla v)^{m} \right)^{2}.
\end{align}
The second line of Eq.~\eqref{expanded_product} follows from the observation that the sum of the squares is less than the square of the sum in cases where all the  terms are non-negative. 

Now, it remains for us to bound the first term in the integrand of Eq.~\eqref{iso_lower_two}---i.e.~the argument of the minimum function
\begin{align}
    \sum_{i=1}^{d+1} \sum_{j=1}^{i-1} \left( p_{K,ij}^{m} \right)^2.
    \label{term_to_be_bounded}
\end{align}
Bounding this term requires a fairly lengthy process, which leverages various geometric properties of the simplex in higher dimensions. With this in mind, we recall that the Levi-Civita symbol $\epsilon_{a_1,a_2,\dots, a_d}$ is completely antisymmetric, that is
\begin{align}
\epsilon_{a_1, a_2, \dots a_d} 
=\begin{cases}
    +1 & \text{if $(a_1, a_2, \dots, a_d)$ is an even permutation of $(1, 2, \dots, d)$}, \\
    -1 & \text{if $(a_1, a_2, \dots, a_d)$ is an odd permutation of $(1, 2, \dots, d)$}, \\
    0 & \text{otherwise}.
\end{cases}
\end{align}
By the definition of the determinant, we have that
\begin{align}
\label{cross_product_determinant}
\det( [\bm{q}_1, \bm{q}_2, \dots, \bm{q}_d]) &= \det \left( \begin{bmatrix}
    q_{1}^{1} & q_{2}^{1} & \cdots & q_{d}^{1} \\[1.0ex]
    q_{1}^{2} & q_{2}^{2} & \cdots & q_{d}^{2} \\[1.0ex]
    & & \vdots & \\
    q_{1}^{m} & q_{2}^{m} & \cdots & q_{d}^{m} \\[1.0ex]
    & & \vdots & \\
    q_{1}^{d} & q_{2}^{d} & \cdots & q_{d}^{d}
\end{bmatrix} \right)
\\
\nonumber
&= (-1)^{m+1} \det \left( \begin{bmatrix}
    q_{1}^{m} & q_{2}^{m} & \cdots & q_{d}^{m} \\[1.0ex]
    q_{1}^{1} & q_{2}^{1} & \cdots & q_{d}^{1} \\[1.0ex]
    & & \vdots & \\
q_{1}^{m-1} & q_{2}^{m-1} & \cdots & q_{d}^{m-1} \\[1.0ex]
    q_{1}^{m+1} & q_{2}^{m+1} & \cdots & q_{d}^{m+1} \\[1.0ex]
    & & \vdots & \\
    q_{1}^{d} & q_{2}^{d} & \cdots & q_{d}^{d}
\end{bmatrix} \right) \\[1.0ex]
\nonumber
&= \epsilon_{a_1, a_2, \dots, a_d}  q_{1}^{a_1} q_{2}^{a_2} \dots q_{d}^{a_d},
\end{align}
where $[\bm{q}_1, \bm{q}_2, \dots, \bm{q}_d]$ denotes the matrix whose columns are $\bm{q}_1, \bm{q}_2, \dots, \bm{q}_d$, and each $q_i^{a_i}$ denotes the $a_i$-th entry of $\bm{q}_i \in \mathbb{R}^d$. Here, we use the Einstein summation convention, that is we sum over repeated indices. Thanks to the standard interpretation of the determinant as an (orientated) hypervolume of a $d$-parallelepiped spanned by the vectors $\bm{q}_1, \bm{q}_{2}, \dots, \bm{q}_d$, we have that
\begin{align}
d! |K| &= |\det( [\bm{q}_1, \bm{q}_2, \dots, \bm{q}_d]) |, \label{volume_identity}
\end{align}
where $|K|$ denotes the hypervolume of the simplex whose edges emanating from $\bm{0}$ are $\bm{q}_1, \bm{q}_2, \dots, \bm{q}_d$. The factor of $d!$ in Eq.~\eqref{volume_identity} originates from the observation that a $d$-cube can be subdivided into $d!$ simplices that all have the same hypervolume. 

Next, we can introduce a normal vector
\begin{align*}
    \bm{n}(r) &= (-1)^{r} \left( \bm{e}_{1} \det \left( \begin{bmatrix} q_{1}^{2} & q_{2}^{2} & \cdots & \widehat{q_{r}^{2}} & \cdots & q_{d}^{2} \\[1.0ex]
    q_{1}^{3} & q_{2}^{3} & \cdots & \widehat{q_{r}^{3}} & \cdots & q_{d}^{3} \\[1.0ex]
    & & \vdots & & \\[1.0ex]
    q_{1}^{m} & q_{2}^{m} & \cdots & \widehat{q_{r}^{m}} & \cdots & q_{d}^{m} \\[1.0ex]
    & & \vdots & & \\[1.0ex]
    q_{1}^{d} & q_{2}^{d} & \cdots & \widehat{q_{r}^{d}} & \cdots & q_{d}^{d}
    \end{bmatrix}  \right) \right. \\[1.0ex] &- \bm{e}_{2} \det \left( \begin{bmatrix} q_{1}^{1} & q_{2}^{1} & \cdots & \widehat{q_{r}^{1}} & \cdots & q_{d}^{1} \\[1.0ex]
    q_{1}^{3} & q_{2}^{3} & \cdots & \widehat{q_{r}^{3}} & \cdots & q_{d}^{3} \\[1.0ex]
    & & \vdots & & \\[1.0ex]
    q_{1}^{m} & q_{2}^{m} & \cdots & \widehat{q_{r}^{m}} & \cdots & q_{d}^{m} \\[1.0ex]
    & & \vdots & & \\[1.0ex]
    q_{1}^{d} & q_{2}^{d} & \cdots & \widehat{q_{r}^{d}} & \cdots & q_{d}^{d}
    \end{bmatrix}  \right) + \cdots \\[1.0ex]
    &+ (-1)^{m+1} \bm{e}_{m} \det \left( \begin{bmatrix} q_{1}^{1} & q_{2}^{1} & \cdots & \widehat{q_{r}^{1}} & \cdots & q_{d}^{1} \\[1.0ex]
    & & \vdots & & \\[1.0ex]
    q_{1}^{m-1} & q_{2}^{m-1} & \cdots & \widehat{q_{r}^{m-1}} & \cdots & q_{d}^{m-1} \\[1.0ex]
    q_{1}^{m+1} & q_{2}^{m+1} & \cdots & \widehat{q_{r}^{m+1}} & \cdots & q_{d}^{m+1} \\[1.0ex]
    & & \vdots & & \\[1.0ex]
    q_{1}^{d} & q_{2}^{d} & \cdots & \widehat{q_{r}^{d}} & \cdots & q_{d}^{d}
    \end{bmatrix}  \right) + \cdots \\[1.0ex]
    &\left. +(-1)^{d+1} \bm{e}_{d} \det \left( \begin{bmatrix} q_{1}^{1} & q_{2}^{1} & \cdots & \widehat{q_{r}^{1}} & \cdots & q_{d}^{1} \\[1.0ex]
    q_{1}^{2} & q_{2}^{2} & \cdots & \widehat{q_{r}^{2}} & \cdots & q_{d}^{2} \\[1.0ex]
    & & \vdots & & \\[1.0ex]
    q_{1}^{m} & q_{2}^{m} & \cdots & \widehat{q_{r}^{m}} & \cdots & q_{d}^{m} \\[1.0ex]
    & & \vdots & & \\[1.0ex]
    q_{1}^{d-1} & q_{2}^{d-1} & \cdots & \widehat{q_{r}^{d-1}} & \cdots & q_{d}^{d-1}
    \end{bmatrix}  \right) \right),
\end{align*}
where $\bm{n}(r)$ is a vector whose magnitude equals the hypervolume of the $(d-1)$-parallelepiped spanned by $\bm{q}_1, \bm{q}_2, \ldots, \widehat{\bm{q}_r}, \ldots, \bm{q}_d$, the hat  $\widehat{\cdot}$ denotes omission, and each $\bm{e}_{m} \in \mathbb{R}^d$ is a vector with 1 in the $m$-th entry and zeros elsewhere. Each component of $\bm{n}(r)$ can be written as follows
\begin{align}
n(r)^{m} = \delta^{a_r, m} \epsilon_{a_1, a_2,\dots, a_d}  q_{1}^{a_1} q_{2}^{a_2} \dots \widehat{q_{r}^{a_r}} \dots q_d^{a_d}, \label{normal_vector_pre}
\end{align}
where $\delta$ denotes the Kronecker delta.
Eq.~\eqref{normal_vector_pre} holds because $\bm{n}(r)$ is orthogonal by construction to all $\bm{q}_i$ and
\begin{align}
| \bm{n}(r)| &=\frac{|\bm{n}(r)| ^2 }{| \bm{n}(r)|} =  \frac{n(r)^{m} \, n(r)^{m}}{| \bm{n}(r)|}  \\
\nonumber &= \epsilon_{a_1, a_2, \dots, a_{r-1},  m, a_{r+1},  \dots, a_d} q_1^{a_1} q_2^{a_2} \dots q_{r-1}^{a_{r-1}} \frac{n(r)^{m}}{| \bm{n}(r)|} q_{r+1}^{a_{r+1} } \dots q_{d}^{a_d},
\end{align}
is the hypervolume of the $d$-parallelepiped spanned by $\bm{q}_1, \bm{q}_2, \ldots, \bm{q}_{r-1}, \frac{\bm{n}(r) }{| \bm{n}(r)|}, \bm{q}_{r+1}, \ldots, \bm{q}_{d}$, which is in turn equal to the hypervolume of the $d$-parallelepiped spanned by $\bm{q}_1, \bm{q}_2, \ldots, \bm{q}_{r-1}, \bm{q}_{r+1}, \ldots, \bm{q}_{d}$ by orthogonality. It follows that
\begin{align}
| \bm{n}(r)| = (d-1)! |\mathcal{F}_{r}|,
\end{align}
where $\mathcal{F}_r$ is the facet opposite $\bm{q}_r$.
%I believe that somehow we should make clear that this construction is not new. I thought that some reference to Hodge duality would make that clear, it predates the previous reference by quite a bit. I don't particularly care about for Hodge theory, but somehow. DW: this is an excellent point. I agree.  
We point out that this construction is not new. Indeed, for the reader that is familiar with Hodge theory, we note that the vectors $\bm{n}(r)$ are the dual vectors of the Hodge dual (the image of the Hodge star operator) of the exterior product of $\bm{q}_{1}^{*}, \dots, \widehat{\bm{q}_{r}^{*}} ,\dots , \bm{q}_{d}^{*}$, where $\bm{w}^{*}$ denotes the dual of $\bm{w}$, (i.e.~the covector of $\bm{w}$). A pedagogical introduction to the concept of Hodge duals appears in section 7.2 of~\cite{abraham2012manifolds}.
%[I am not sure that this is the best reference it is low dimensional, but it was surprisingly hard to find a reference for something that is rather classical, I've written a colleague to ask if he has a good reference. DW: Sounds good. I will let you know if I think of a better reference.]

Let us return our attention to the vectors themselves: $\bm{q}_1, \bm{q}_2, \ldots, \widehat{\bm{q}_r}, \ldots, \bm{q}_d$. These vectors can be defined so that
\begin{align*}
    \bm{q}_1 &= \bm{p}_{K,1}-\bm{p}_{K,\ell} = \bm{p}_{K,\ell 1}, \\
    \bm{q}_2 &= \bm{p}_{K,2}-\bm{p}_{K,\ell} = \bm{p}_{K,\ell 2}, \\
    &\vdots \\
    \widehat{\bm{q}_r} &= \widehat{\bm{p}_{K,r}-\bm{p}_{K,\ell}} = \widehat{\bm{p}_{K,\ell r}} \\
    &\vdots \\
    \bm{q}_{\ell-1} & = \bm{p}_{K,\ell-1}-\bm{p}_{K,\ell} = \bm{p}_{K,\ell (\ell-1)},\\
    \bm{q}_{\ell} & = \bm{p}_{K,\ell+1}-\bm{p}_{K,\ell} = \bm{p}_{K,\ell (\ell+1)}, \\
    &\vdots \\
    \bm{q}_d &= \bm{p}_{K,d+1}-\bm{p}_{K,\ell} = \bm{p}_{K,\ell (d+1)}.
\end{align*}
Here, we have shifted each vertex by $\bm{p}_{K,\ell}$. Evidently, this ensures that the vertex $\bm{p}_{K,\ell}$ itself is shifted to the origin. As a result, we have that
\begin{align}
| \bm{n}(r)| = (d-1)! |\mathcal{F}_{K,r}|,
    \label{facet_volume_identity}
\end{align}
where $\mathcal{F}_{K,r}$ is the facet opposite the vertex $\bm{p}_{K,r}$. 

Next, in accordance with Eqs.~\eqref{cross_product_determinant}, \eqref{volume_identity}, and \eqref{normal_vector_pre}, we have that
\begin{align}
    \left| K \right| \leq \frac{1}{d!} \sum_{\substack{r=1\\ r\neq \ell}}^{d+1} \left( \left| p_{K,\ell r}^{m} \right| \left| n(r)^{m} \right| \right). \label{volume_identity_expanded_one}
\end{align}
Upon summing Eq.~\eqref{volume_identity_expanded_one} over all $\ell$, we obtain
\begin{align}
        \label{volume_identity_expanded_two} (d+1)\left| K \right| &\leq \frac{1}{d!} \sum_{\ell = 1}^{d+1} \sum_{\substack{r=1\\ r\neq \ell}}^{d+1} \left( \left| p_{K,\ell r}^{m} \right| \left| n(r)^{m} \right| \right) \\[1.0ex]
        \nonumber &\leq \frac{1}{d!} \sum_{\ell = 1}^{d+1} \sum_{\substack{r=1 \\ r\neq \ell}}^{d+1} \left( \left| p_{K,\ell r}^{m} \right| \left| \bm{n}(r) \right| \right) \\[1.0ex]
        \nonumber &= \frac{1}{d} \sum_{\ell = 1}^{d+1} \sum_{\substack{r=1 \\ r\neq \ell}}^{d+1} \left( \left| p_{K,\ell r}^{m} \right| \left| \mathcal{F}_{K,r} \right|  \right) \\[1.0ex]
        \nonumber &\leq \frac{1}{d} \left( \max_{s} |\mathcal{F}_{K,s}| \right) \sum_{\ell = 1}^{d+1} \sum_{\substack{r=1 \\ r\neq \ell}}^{d+1} \left| p_{K,\ell r}^{m} \right|, \qquad s \in [1, d+1],
\end{align}
where we have used Eq.~\eqref{facet_volume_identity} on the second-to-last line of Eq.~\eqref{volume_identity_expanded_two}.

We note that the following identity holds by a symmetry argument
\begin{align}
    \label{edge_identity}
    \frac{1}{2} \sum_{\ell = 1}^{d+1} \sum_{\substack{r=1 \\ r\neq \ell}}^{d+1}  \left| p_{K,\ell r}^{m} \right| =  \sum_{i=1}^{d+1} \sum_{j=1}^{i-1}  \left| p_{K,ij}^{m} \right|.
\end{align}
Substituting Eq.~\eqref{edge_identity} into Eq.~\eqref{volume_identity_expanded_two} yields
\begin{align}
    \frac{d(d+1)}{2} \frac{\left| K \right|}{ \left( \max_{s} |\mathcal{F}_{K,s}| \right)} \leq \sum_{i=1}^{d+1} \sum_{j=1}^{i-1}  \left| p_{K,ij}^{m} \right|.
    \label{first_clean_bound}
\end{align}
The ratio $|K|/\max_{s}|\mathcal{F}_{K,s}|$ on the left hand side (above) can be rewritten in terms of the minimum elevation of the simplex. In particular
\begin{align}
    \frac{\left| K \right|}{ \left( \max_{s} |\mathcal{F}_{K,s}| \right)} = \frac{\min_{s} \left[\mathrm{dist}\left(\bm{p}_{K,s}, \mathrm{aff}(\mathcal{F}_{K,s})\right)\right]}{d}, 
    \label{elevation_identity}
\end{align}
where the function $\mathrm{dist}(\cdot,\cdot)$ returns the shortest distance between the vertex $\bm{p}_{K,s}$ and the affine hull of its opposite facet, $\mathrm{aff}(\mathcal{F}_{K,s})$. We can substitute Eq.~\eqref{elevation_identity} into Eq.~\eqref{first_clean_bound} as follows
\begin{align}
    \frac{(d+1)}{2} \min_{s} \left[\mathrm{dist}\left(\bm{p}_{K,s}, \mathrm{aff}(\mathcal{F}_{K,s})\right)\right] \leq \sum_{i=1}^{d+1} \sum_{j=1}^{i-1}  \left| p_{K,ij}^{m} \right|.
\end{align}
Next, upon squaring both sides of the expression above
\begin{align}
    \left(\frac{(d+1)}{2} \min_{s} \left[\mathrm{dist}\left(\bm{p}_{K,s}, \mathrm{aff}(\mathcal{F}_{K,s})\right)\right]\right)^2 \leq \left(\sum_{i=1}^{d+1} \sum_{j=1}^{i-1}  \left| p_{K,ij}^{m} \right|\right)^2 \leq \frac{d(d+1)}{2} \sum_{i=1}^{d+1} \sum_{j=1}^{i-1}  \left( p_{K,ij}^{m} \right)^2. \label{simple_volume_one}
\end{align}
Here, we have leveraged the root-mean-square-arithmetic-mean inequality in order to reformulate the right hand side.

Equivalently, upon simplifying Eq.~\eqref{simple_volume_one}, we obtain
\begin{align}
   \frac{d+1}{2d} \left( \min_{s} \left[\mathrm{dist}\left(\bm{p}_{K,s}, \mathrm{aff}(\mathcal{F}_{K,s})\right)\right] \right)^2 \leq  \sum_{i=1}^{d+1} \sum_{j=1}^{i-1}  \left( p_{K,ij}^{m} \right)^2. \label{simple_volume_two}
\end{align}
Finally, substituting Eq.~\eqref{simple_volume_two} into Eq.~\eqref{iso_lower_two}, and setting $h_K = \Delta(K)$, yields
\begin{align}
    \label{iso_lower_three}
      \left(\Psi_{\mathcal{T}}\left(\nabla v\right)\right)^{2} &\geq \sum_{K \in \mathcal{T}} \frac{1}{\Delta(K)^{2}} \int_{K} \min_{m} \left(\sum_{i=1}^{d+1} \sum_{j=1}^{i-1} \left( p_{K,ij}^{m} \right)^2 \right) \left( \nabla v\cdot \nabla v \right) dV \\[1.0ex]
    &\geq \nonumber \frac{d+1}{2d} \sum_{K \in \mathcal{T}} \frac{1}{\Delta(K)^{2}} \int_{K} \min_{m} \left( \min_{s} \left[\mathrm{dist}\left(\bm{p}_{K,s}, \mathrm{aff}(\mathcal{F}_{K,s})\right)\right] \right)^2 \left( \nabla v\cdot \nabla v \right) dV 
     \\[1.0ex]
    &\geq \nonumber \frac{d+1}{2d} \min_{K} \left( \frac{\min_{s} \left[\mathrm{dist}\left(\bm{p}_{K,s}, \mathrm{aff}(\mathcal{F}_{K,s})\right)\right]}{\Delta(K)}  \right)^2 \sum_{K \in \mathcal{T}} \int_{K} (\nabla v \cdot \nabla v) \, dV  \\[1.0ex]
    &= \nonumber \frac{d+1}{2d}  \min_{K} \left( \frac{\min_{s} \left[\mathrm{dist}\left(\bm{p}_{K,s}, \mathrm{aff}(\mathcal{F}_{K,s})\right)\right]}{\Delta(K)} \right)^2 \left\| \nabla v \right\|_{L^{2}(\Omega)}^{2}.
\end{align}
Upon setting
\begin{align}
    \label{sliver_constant_orig}
    C_{1} \equiv \sqrt{\frac{d+1}{2d}} \min_{K} \left( \frac{\min_{s} \left[\mathrm{dist}\left(\bm{p}_{K,s}, \mathrm{aff}(\mathcal{F}_{K,s})\right)\right]}{\Delta(K)} \right),
\end{align}
in Eq.~\eqref{iso_lower_three}, and taking the square root of both sides, we obtain the desired lower bound for $\Psi_{\mathcal{T}}\left(\nabla v\right)$.

Next, we will construct the upper bound for the roughness functional and identify the constant $C_2$. From the definition of the  functional, we have
\begin{align}
    \label{iso_upper_one}
    \left(\Psi_{\mathcal{T}}\left(\nabla v\right)\right)^{2} &= \sum_{K \in \mathcal{T}}  \frac{1}{h_{K}^{2}} \int_{K} \sum_{i=1}^{d+1} \sum_{j=1}^{i-1}\left(\mathrm{abs}(\nabla v)\cdot \mathrm{abs}(\bm{p}_{K,ij}) \right)^{2} dV \\[1.0ex]
    \nonumber &\leq \sum_{K \in \mathcal{T}}  \frac{1}{h_{K}^{2}} \int_{K} \sum_{i=1}^{d+1} \sum_{j=1}^{i-1}\left(\mathrm{abs}(\nabla v)\cdot \mathrm{abs}(\nabla v) \right)\left(\mathrm{abs}(\bm{p}_{K,ij})\cdot \mathrm{abs}(\bm{p}_{K,ij}) \right) dV
    \\[1.0ex]
    \nonumber &= \sum_{K \in \mathcal{T}}  \frac{1}{h_{K}^{2}} \int_{K} \left(\nabla v\cdot \nabla v \right)\sum_{i=1}^{d+1} \sum_{j=1}^{i-1}\left(\bm{p}_{K,ij}\cdot \bm{p}_{K,ij} \right) dV
    \\[1.0ex]
    \nonumber &= \sum_{K \in \mathcal{T}}  \frac{1}{h_{K}^{2}} \sum_{i=1}^{d+1} \sum_{j=1}^{i-1}\left(\bm{p}_{K,ij}\cdot \bm{p}_{K,ij} \right) \int_{K} \left(\nabla v\cdot \nabla v \right) dV
    \\[1.0ex]
    \nonumber &\leq \max_{K} \left( \frac{1}{h_{K}^{2}} \sum_{i=1}^{d+1} \sum_{j=1}^{i-1}\left(\bm{p}_{K,ij}\cdot \bm{p}_{K,ij} \right) \right) \sum_{K \in \mathcal{T}} \int_{K} \left(\nabla v\cdot \nabla v \right) dV
    \\[1.0ex]
    \nonumber &= \max_{K} \left( \frac{1}{\Delta(K)^{2}} \sum_{i=1}^{d+1} \sum_{j=1}^{i-1}\left(\bm{p}_{K,ij}\cdot \bm{p}_{K,ij} \right) \right) \left\| \nabla v \right\|_{L_{2}(\Omega)}^{2},
\end{align}
where the Cauchy-Schwarz inequality has been used on the second line, and $h_K = \Delta(K)$ has been used on the last line. Upon setting
\begin{align}
    C_{2} \equiv \max_{K} \left( \frac{1}{\Delta(K)} \sqrt{\sum_{i=1}^{d+1} \sum_{j=1}^{i-1}\left(\bm{p}_{K,ij}\cdot \bm{p}_{K,ij} \right) }\right),
    \label{coarse_constant_orig}
\end{align}
in Eq.~\eqref{iso_upper_one} and taking the square root of both sides, we obtain the desired upper bound for $\Psi_{\mathcal{T}}\left(\nabla v\right)$.
\end{proof}

\subsection{Proof of Lemma~\ref{iso_upper_bound_lemma}}
\label{iso_upper_bound_lemma_proof}

\begin{proof}
    Consider the definition of the roughness functional
    \begin{align}
    \label{iso_upper_two}
        \left(\Psi_{\mathcal{T}}\left(\nabla v\right)\right)^{2} &= 
         \sum_{K \in \mathcal{T}} \frac{1}{h_{K}^{2}} \int_{K} \sum_{i=1}^{d+1} \sum_{j=1}^{i-1}\left(\mathrm{abs}(\nabla v)\cdot \mathrm{abs}(\bm{p}_{K,ij}) \right)^{2} dV
        \\[1.0ex]
    \nonumber &\leq \sum_{K \in \mathcal{T}} \frac{1}{h_{K}^{2}}\sum_{i=1}^{d+1} \sum_{j=1}^{i-1}\left(\bm{p}_{K,ij}\cdot \bm{p}_{K,ij} \right) \int_{K} \left(\nabla v\cdot \nabla v \right) dV
        \\[1.0ex]
        \nonumber &\leq \sum_{K \in \mathcal{T}} \frac{1}{h_{K}^{2}} \sum_{i=1}^{d+1} \sum_{j=1}^{i-1}\left(\bm{p}_{K,ij}\cdot \bm{p}_{K,ij} \right) |K| \left\|\nabla v\cdot \nabla v \right\|_{L_{\infty}(K)} 
        \\[1.0ex]
        \nonumber &\leq \sum_{K \in \mathcal{T}} \left( \frac{1}{h_{K}^{2}} \sum_{i=1}^{d+1} \sum_{j=1}^{i-1}\left(\bm{p}_{K,ij}\cdot \bm{p}_{K,ij} \right) |K| \right) \left\|\nabla v\cdot \nabla v \right\|_{L_{\infty}(\Omega)}
        \\[1.0ex]
        \nonumber &= \sum_{K \in \mathcal{T}} \left( \frac{1}{\Delta(K)^{2}} \sum_{i=1}^{d+1} \sum_{j=1}^{i-1}\left(\bm{p}_{K,ij}\cdot \bm{p}_{K,ij} \right) |K| \right) \left\|\nabla v \right\|_{L_{\infty}(\Omega)}^{2},
    \end{align}
    where we have set $h_K = \Delta(K)$ on the last line. It remains for us to analyze the term 
    \begin{align*}
    \frac{1}{\Delta(K)^2},
    \end{align*}
    that appears on the right hand side (above). 

    In accordance with Eq.~\eqref{size_constraint_two}, we have that
    \begin{align}
        \label{length_scale_one}
         \frac{1}{\Delta(K)^{2}} = \mathcal{L}_{[1/\mathcal{D}^2]} (\bm{x}_{K}) - \zeta_{K} &=  \mathcal{L}_{[1/\mathcal{D}^2]} (\bm{x}_{K}) - \frac{1}{\mathcal{D}^{2}(\bm{x}_K)} + \frac{1}{\mathcal{D}^{2}(\bm{x}_K)} -  \zeta_{K} \\[1.0ex]
         \nonumber &\leq \left| \mathcal{L}_{[1/\mathcal{D}^2]} (\bm{x}_{K}) - \frac{1}{\mathcal{D}^{2}(\bm{x}_K)} \right| + \left| \frac{1}{\mathcal{D}^{2}(\bm{x}_K)} \right| + \left| \zeta_{K} \right| \\[1.0ex]
         \nonumber &\leq \left\| \mathcal{L}_{[1/\mathcal{D}^2]}  - \frac{1}{\mathcal{D}^{2}} \right\|_{L_{\infty}(K)} + \left\| \frac{1}{\mathcal{D}^{2}} \right\|_{L_{\infty}(K)} + \left| \zeta_{K} \right|.
    \end{align}
    The first term on the right hand side of Eq.~\eqref{length_scale_one} can be bounded as follows (see Waldron~\cite{waldron1998error}, Theorem 3.1)
    \begin{align}
        \label{length_scale_two}
        \left\| \mathcal{L}_{[1/\mathcal{D}^2]} - \frac{1}{\mathcal{D}^{2}} \right\|_{L_{\infty}(K)} \leq \frac{1}{2} R_{K,\mathrm{min}}^{2} \left\| \frac{1}{\mathcal{D}^{2}} \right\|_{2,L_{\infty}(K)},
    \end{align}
    where $R_{K,\mathrm{min}}$ is the min-containment radius of $K$,
    \begin{align*}
        \left\| \frac{1}{\mathcal{D}^{2}} \right\|_{2,L_{\infty}(K)} \equiv \left\| \, \left| \frac{\partial^{2}}{\partial \bm{u}^2} \left( \frac{1}{\mathcal{D}^{2}}\right) \right| \, \right\|_{L_{\infty}(K)},
    \end{align*}
    and where
    \begin{align*}
        \left| \frac{\partial^{2}}{\partial \bm{u}^2} \left( \frac{1}{\mathcal{D}^{2}}\right) \right|(\bm{x}) \equiv \sup_{\substack{\bm{u}_{1}, \bm{u}_{2} \in \mathbb{R}^d \\ \left\| \bm{u}_{i} \right\| \leq 1}} \left| \frac{\partial}{\partial \bm{u}_{1}} \frac{\partial}{\partial \bm{u}_{2}} \left( \frac{1}{\mathcal{D}^{2}(\bm{x})}\right) \right|.
    \end{align*}
    We may now substitute Eq.~\eqref{length_scale_two} into Eq.~\eqref{length_scale_one} in order to obtain
    \begin{align}
        \frac{1}{\Delta(K)^{2}} \leq \frac{1}{2} R_{K,\mathrm{min}}^{2} \left\| \frac{1}{\mathcal{D}^{2}} \right\|_{2,L_{\infty}(K)} + \left\| \frac{1}{\mathcal{D}^{2}} \right\|_{L_{\infty}(K)} + \left| \zeta_{K} \right|.
    \end{align}
    Evidently, we can immediately generalize this result to the entire mesh
    \begin{align}
        \max_{K} \left(\frac{1}{\Delta(K)^{2}} \right) \leq \frac{1}{2} \max_{K} \left( R_{K,\mathrm{min}}\right)^{2}  \left\| \frac{1}{\mathcal{D}^{2}} \right\|_{2,L_{\infty}(\Omega)} + \left\| \frac{1}{\mathcal{D}^{2}} \right\|_{L_{\infty}(\Omega)} + \max_{K} \left| \zeta_{K} \right|. \label{max_upper_bound_diam}
    \end{align}
    Interestingly enough, because the vertices of our mesh are an $(\varepsilon,\overline{\eta})$-net, we also have that
    \begin{align}
        \max_{K} \left(\frac{1}{\Delta(K)^{2}} \right) \leq \frac{1}{\eta^2}.
        \label{max_upper_bound_aux}
    \end{align}
    We now return our attention to Eq.~\eqref{iso_upper_two}. Upon substituting Eqs.~\eqref{max_upper_bound_diam} and \eqref{max_upper_bound_aux} into Eq.~\eqref{iso_upper_two}, one obtains
    \begin{align}
        \label{iso_upper_three}\left(\Psi_{\mathcal{T}}\left(\nabla v\right)\right)^{2} &\leq \min\left[\frac{1}{2} \max_{K} \left( R_{K,\mathrm{min}}\right)^{2}  \left\| \frac{1}{\mathcal{D}^{2}} \right\|_{2,L_{\infty}(\Omega)} + \left\| \frac{1}{\mathcal{D}^{2}} \right\|_{L_{\infty}(\Omega)} + \max_{K} \left| \zeta_{K} \right|, \frac{1}{\eta^2} \right] \\[1.0ex]
        \nonumber &\times \sum_{K \in \mathcal{T}} \left( \sum_{i=1}^{d+1} \sum_{j=1}^{i-1}\left(\bm{p}_{K,ij}\cdot \bm{p}_{K,ij} \right) |K| \right) \left\|\nabla v \right\|_{L_{\infty}(\Omega)}^{2}. 
    \end{align}
    Finally, on setting 
    \begin{align}
        C_{3} &\equiv \sqrt{\min \left[\frac{1}{2} \max_{K} \left( R_{K,\mathrm{min}}\right)^{2}  \left\| \frac{1}{\mathcal{D}^{2}} \right\|_{2,L_{\infty}(\Omega)} + \left\| \frac{1}{\mathcal{D}^{2}} \right\|_{L_{\infty}(\Omega)} + \max_{K} \left| \zeta_{K} \right| , \frac{1}{\eta^2} \right]}, \\[1.0ex]
        \Theta &\equiv \sum_{K \in \mathcal{T}} \left(\sum_{i=1}^{d+1} \sum_{j=1}^{i-1}\left(\bm{p}_{K,ij}\cdot \bm{p}_{K,ij} \right) |K| \right),
    \end{align}
    in Eq.~\eqref{iso_upper_three} and taking the square root of both sides, one obtains the desired result. 
\end{proof}

\subsection{Proof of Theorem~\ref{functional_error_estimate_lemma}}
\label{functional_error_estimate_lemma_proof}

\begin{proof}
    We start by employing the Triangle inequality as follows
    \begin{align}
        \label{iso_error_one}
        \left\| \frac{\partial v}{\partial x_i} - \left(\frac{\partial v}{\partial x_i} \right)_h \right\|_{L_{\infty}(K)} &= \left\| \frac{\partial v}{\partial x_i} - \left(\frac{\partial v}{\partial x_i} \right)^{\ast} + \left(\frac{\partial v}{\partial x_i} \right)^{\ast} - \left(\frac{\partial v}{\partial x_i} \right)_h \right\|_{L_{\infty}(K)} \\[1.0ex]
        \nonumber &\leq \left\| \frac{\partial v}{\partial x_i} - \left(\frac{\partial v}{\partial x_i} \right)^{\ast} \right\|_{L_{\infty}(K)} + \left\| \left(\frac{\partial v}{\partial x_i} \right)^{\ast} - \left(\frac{\partial v}{\partial x_i} \right)_h \right\|_{L_{\infty}(K)}.
    \end{align}
    Here, the quantity
    \begin{align*}
         \left(\frac{\partial v}{\partial x_i} \right)^{\ast}
    \end{align*}
    is the `best approximation' of the gradient, (see~\cite[chapter 10]{trefethen2019approximation} for details). Note that $\left(\frac{\partial v}{\partial x_i} \right)^{\ast}$ and $\left(\frac{\partial v}{\partial x_i} \right)_h$ reside in the same piecewise polynomial space, and that $\left(\frac{\partial v}{\partial x_i} \right)^{\ast}$ minimizes the error in the infinity norm.
    
    In accordance with classical interpolation theory, see for example~\cite[Theorem 15.1]{trefethen2019approximation}, or~\cite[pg.~47]{hesthaven2007nodal}, the following inequality holds
    \begin{align}
        \label{iso_error_two}
        \left\| \left(\frac{\partial v}{\partial x_i} \right)^{\ast} - \left(\frac{\partial v}{\partial x_i} \right)_h \right\|_{L_{\infty}(K)} \leq \Lambda \left\| \frac{\partial v}{\partial x_i} - \left(\frac{\partial v}{\partial x_i} \right)^{\ast} \right\|_{L_{\infty}(K)},
    \end{align}
    where
    \begin{align}
        \Lambda \equiv \max_{\bm{\xi}}\sum_{m=1}^{N_p}|L_m(\bm{\xi})|,
    \end{align}
    is the well-known Lebesgue constant. Upon substituting Eq.~\eqref{iso_error_two} into Eq.~\eqref{iso_error_one}, we obtain
    \begin{align}
        \label{iso_error_three}
        \left\| \frac{\partial v}{\partial x_i} - \left(\frac{\partial v}{\partial x_i} \right)_h \right\|_{L_{\infty}(K)} \leq (1+ \Lambda) \left\| \frac{\partial v}{\partial x_i} - \left(\frac{\partial v}{\partial x_i} \right)^{\ast} \right\|_{L_{\infty}(K)}.
    \end{align}
    Next, summing Eq.~\eqref{iso_error_three} over the number of dimensions $d$ yields
    \begin{align}
        \label{iso_error_four}
        \left\| \nabla v - (\nabla v)_h \right\|_{L_{\infty}(K)} \leq (1 + \Lambda)\left\| \nabla v - (\nabla v)^{\ast} \right\|_{L_{\infty}(K)}.
    \end{align}
    Evidently
    \begin{align}
        \label{iso_error_five}
        \left\| \nabla v - (\nabla v)_h \right\|_{L_{\infty}(\Omega)} \leq (1 + \Lambda)\left\| \nabla v - (\nabla v)^{\ast} \right\|_{L_{\infty}(\Omega)}.
    \end{align}
    Now, we can replace $\nabla v$ with $(\nabla v - (\nabla v)_h)$ in Eq.~\eqref{iso_upper_bound} from Lemma~\ref{iso_upper_bound_lemma}, as follows
    \begin{align}
        \Psi_{\mathcal{T}}\left(\nabla v - (\nabla v)_h\right) \leq  C_3 \sqrt{\Theta} \left\| \nabla v - (\nabla v)_h \right\|_{L_{\infty}(\Omega)}
        \label{upper_bound_pre}
    \end{align}
    Upon substituting Eq.~\eqref{iso_error_five} into the right hand side of Eq.~\eqref{upper_bound_pre}, we obtain the desired result.
\end{proof}

\end{document}